\documentclass[a4paper,12pt]{article}
\usepackage{tikz}
\usetikzlibrary{matrix,arrows,decorations.markings}
\usepackage{amsmath,amsfonts,amssymb,amsthm,url,textcomp}
\usepackage{csquotes}
\usepackage{a4wide}
\usepackage[numbers]{natbib}
\usepackage{fancyhdr}
\usepackage{anyfontsize}
\usepackage{hyperref}
\usepackage{hhline}
\usepackage{leftidx}
\usepackage{enumitem}
\usepackage{mdframed}
\usepackage[ruled,vlined]{algorithm2e}
\usepackage{makecell}

\allowdisplaybreaks

\makeatletter
\let\@fnsymbol\@arabic
\makeatother

\newtheorem{theorem}{Theorem}\numberwithin{theorem}{section}
\newtheorem{definition}[theorem]{Definition}
\newtheorem*{definition*}{Definition}
\newtheorem*{proposition*}{Proposition}
\newtheorem{lemma}[theorem]{Lemma}
\newtheorem{corollary}[theorem]{Corollary}
\newtheorem{proposition}[theorem]{Proposition}
\newtheorem{notation}[theorem]{Notation}

\numberwithin{theoremm}{subsection}

\numberwithin{theoremmm}{subsubsection}

\theoremstyle{remark}

\newtheorem{example}[theorem]{Example}
\newtheorem*{example*}{Example}

\newcommand{\Aut}{\operatorname{Aut}}

\newcommand{\PSL}{\operatorname{PSL}}

\newcommand{\Aff}{\operatorname{Aff}}

\newcommand{\lcm}{\operatorname{lcm}}

\newcommand{\ord}{\operatorname{ord}}
\newcommand{\Sym}{\operatorname{Sym}}

\newcommand{\id}{\operatorname{id}}

\newcommand{\GL}{\operatorname{GL}}

\newcommand{\im}{\operatorname{im}}

\newcommand{\Comp}{\operatorname{Comp}}

\newcommand{\Mod}[1]{\ (\textup{mod}\ #1)}

\newcommand{\IN}{\mathbb{N}}

\newcommand{\IF}{\mathbb{F}}

\newcommand{\End}{\operatorname{End}}

\newcommand{\IZ}{\mathbb{Z}}

\newcommand{\CT}{\operatorname{CT}}

\newcommand{\SL}{\operatorname{SL}}

\newcommand{\imp}{\mathrm{imp}}

\newcommand{\IQ}{\mathbb{Q}}

\newcommand{\CWAff}{\operatorname{CWAff}}

\newcommand{\BU}{\operatorname{BU}}

\newcommand{\Bcal}{\mathcal{B}}

\newcommand{\CGL}{\operatorname{CGL}}
\newcommand{\AGL}{\operatorname{AGL}}
\newcommand{\ACGL}{\operatorname{ACGL}}
\newcommand{\RowSpace}{\operatorname{RowSpace}}
\newcommand{\OGL}{\operatorname{OGL}}

\begin{document}

\title{Coset-wise affine functions and cycle types of complete mappings}

\author{Alexander Bors\textsuperscript{1} \and Qiang Wang\thanks{School of Mathematics and Statistics, Carleton University, 1125 Colonel By Drive, Ottawa ON K1S 5B6, Canada. \newline First author's e-mail: \href{mailto:alexanderbors@cunet.carleton.ca}{alexanderbors@cunet.carleton.ca} \newline Second author's e-mail: \href{mailto:wang@math.carleton.ca}{wang@math.carleton.ca} \newline The authors are supported by the Natural Sciences and Engineering Research Council of Canada (RGPIN-2017-06410). \newline 2020 \emph{Mathematics Subject Classification}: Primary: 12E20. Secondary: 11T06, 15A21, 20B05. \newline \emph{Keywords and phrases}: Finite field, Complete mapping, Cycle type, Cycle structure, Wreath product.}}

\date{\today}

\maketitle

\abstract{Let $K$ be a finite field of characteristic $p$. We study a certain class of functions $K\rightarrow K$ that agree with an $\IF_p$-affine function $K\rightarrow K$ on each coset of a given additive subgroup $W$ of $K$ -- we call them \emph{$W$-coset-wise $\IF_p$-affine functions of $K$}. We show that these functions form a permutation group on $K$ with the structure of an imprimitive wreath product and characterize which of them are complete mappings of $K$. As a consequence, we are able to provide various new examples of cycle types of complete mappings of $K$, including that $K$ has a complete mapping moving all elements of $K$ in one cycle if $p>2$.}

\section{Introduction}\label{sec1}

\subsection{Background and main results}\label{subsec1P1}

Let $(G,+)$ be an additive (but not necessarily abelian) group. A \emph{complete mapping of $G$} is a permutation $f$ of $G$ such that the function $f+\id:G\rightarrow G, g\mapsto f(g)+g$, is also a permutation of $G$. A \emph{complete mapping of a field $K$} is just a complete mapping of the underlying additive group of $K$. See Evans' book \cite{Eva92a} for a concise introduction to the theory of complete mappings.

The study of complete mappings has a long and rich history and has been spurred by their various applications. Originally, complete mappings were introduced by Mann in 1942 as a tool in constructing mutually orthogonal Latin squares \cite{Man42a}. The question of which groups admit complete mappings, which comprises the celebrated \emph{Hall-Paige Conjecture}, was heavily studied and finally answered completely by group theorists -- in chronological order, this involved the work of Bateman \cite{Bat50a}, Hall-Paige \cite{HP55a}, Wilcox \cite{Wil09a}, Evans \cite{Eva09a}, and Bray \cite[Section 2]{BCCSZ20a}. A unified proof of this conjecture can be found in Evans' book \cite{Eva18a}, an expansion of his other book \cite{Eva92a} cited earlier.

An important and heavily studied special case are complete mappings of finite fields (or, equivalently, of finite elementary abelian groups), whose polynomial representations have been investigated extensively. An influential early work in this regard is Niederreiter-Robinson's 1984 paper \cite{NR84a}. This paper came before the various practical applications of complete mappings were discovered, for example in check-digit systems \cite{Sch00a,SW10a} and the construction of cryptographic functions \cite{MP14a,SGCGM12a}. These applications spurred even greater interest in complete mappings, see e.g.~\cite{ITW17a,TZH14a,Win14a,WLHZ14a,XC15a,ZHC15a}.

Recall that every permutation $\sigma$ of a finite set $\Omega$ decomposes into pairwise disjoint cycles. The \emph{cycle type of $\sigma$}, which we will denote by $\CT(\sigma)$ in this paper, is the unique monomial
\[
x_1^{k_1}x_2^{k_2}\cdots x_{|\Omega|}^{k_{|\Omega|}}\in\IQ[x_n:n\geq 1]
\]
where $k_{\ell}$ is the number of length $\ell$ cycles of $\sigma$ for $\ell=1,2,\ldots,|\Omega|$. Hence, $\CT(\sigma)$ encodes the information how many cycles of each given length $\sigma$ has.

Although complete mappings of finite fields have been heavily studied, it is still not well understood which cycle types can be achieved by them. The following two elementary facts are known:
\begin{enumerate}
\item A complete mapping $f$ of an abelian group $G$ cannot have a $2$-cycle. Indeed, if $(x,f(x))$ is a $2$-cycle of $f$, then $(f+\id)(x)=f(x)+x=x+f(x)=f(f(x))+f(x)=(f+\id)(f(x))$, contradicting the injectivity of $f+\id$.
\item An \emph{orthomorphism} of an additive group $G$ is a permutation $f$ of $G$ such that $f-\id:G\rightarrow G,g\mapsto f(g)-g$, is also a permutation of $G$. Because a fixed point of $f$ is the same as a pre-image of the neutral element $0_G$ under $f-\id$, each orthomorphism of an arbitrary group has precisely one fixed point. Note that if $G$ has exponent $2$ (which happens precisely when $G$ is the underlying additive group of a field of characteristic $2$), then complete mappings of $G$ are the same as orthomorphisms of $G$, and so complete mappings of $G$ must have precisely one fixed point then.
\end{enumerate}
Note that these conditions allow one to refute certain cycle types as possible for complete mappings of a finite field $\IF_q$. On the other hand, we also have some knowledge of cycle types that are possible:
\begin{enumerate}
\item A complete mapping (or orthormorphism) of cycle type $x_1x_{\ell}^{(q-1)/\ell}$ is called \emph{$\ell$-regular} (see \cite{FGT81a}, whence this terminology apparently originated). Because a scalar multiplication $\IF_q\rightarrow\IF_q, x\mapsto ax$ for fixed $a\in\IF_q$, is a complete mapping of $\IF_q$ if and only if $a\notin\{0,-1\}$, we see that $\ell$-regular complete mappings of $\IF_q$ exist for every divisor $\ell$ of $q-1$ such that
\[
\ell\not=\begin{cases}2, & \text{if }2\nmid q, \\ 1, & \text{if }2\mid q.\end{cases}
\]
Other examples of $\ell$-regular complete mappings have been studied by various authors, see \cite{CWZ17a,GGM01a,Mit95a,Mit97a,NKQW20a,WLW20a}.
\item In their earlier paper \cite{BW21a}, the authors derived existence results for cycle types of complete mappings of $\IF_q$ that are \emph{generalized cyclotomic mappings of $\IF_q$} -- functions which fix $0$ and restrict to a monomial function $x\mapsto a_ix^{r_i}$ on each coset $C_i$ of a fixed subgroup $C$ of $\IF_q^{\ast}$.
\item Over a finite field of characteristic $p>2$, complete mappings need not have any fixed points. A trivial example for this are the functions $\IF_q\rightarrow\IF_q$, $x\mapsto x+c$ with $c\in\IF_q^{\ast}$ fixed, which have cycle type $x_p^{q/p}$. For a less trivial class of examples, see \cite[Theorem 9]{MP14a}.
\end{enumerate}

Our goal in this paper is to study an additive analogue of the multiplicative approach of \cite{BW21a}. That is, we will consider the following notion of a function that is defined via additive cosets:
\begin{definition}\label{cosetWiseDef}
Let $K$ be a field, let $V$ be a $K$-vector space, and let $W$ be a $K$-subspace of $V$. A function $f:V\rightarrow V$ is called \emph{$W$-coset-wise $K$-affine} if for each coset $C$ of $W$ in $V$, there is a vector $v_C\in V$ and an endomorphism $\varphi_C\in\End(V)$ with $\varphi_C(W)\subseteq W$ such that for all $x\in C$, one has $f(x)=\varphi_C(x)+v_C$.
\end{definition}
In our application, we will have $V=\IF_q=\IF_{p^k}$ and $K=\IF_p$, but it is more convenient to discuss the problem in terms of vector spaces (rather than fields). Our main result, Theorem \ref{mainTheo} below, is a means of obtaining the existence of cycle types of complete mappings of finite-dimensional $\IF_p$-vector spaces from known cycle types of complete mappings on $\IF_p$-vector spaces of smaller dimension. Before we formulate it, we introduce a bit of notation:
\begin{notation}\label{mainNot}
We introduce the following pieces of notation:
\begin{enumerate}
\item If $X$ is a set of permutations of a given finite set $\Omega$, then we set
\[
\CT(X):=\{\CT(\sigma):\sigma\in X\}.
\]
\item Let $\ell$ be a positive integer. The \emph{$\ell$-blow-up function} is the unique $\IQ$-algebra endomorphism $\BU_{\ell}$ of $\IQ[x_n:n\geq1]$ such that $\BU_{\ell}(x_n)=x_{\ell n}$ for all $n\in\IN^+$.
\item Let $d$ be a positive integer and $q$ be a prime power.
\begin{enumerate}
\item Recall that $\GL_d(q)$ denotes the group of invertible $(d\times d)$-matrices over $\IF_q$.
\item If $M$ is a $(d\times d)$-matrix over $\IF_q$ and $v\in\IF_q^d$ a row vector, then the function $\lambda(M,v):\IF_q^d\rightarrow\IF_q^d$, $x\mapsto xM+v$, is called an \emph{$\IF_q$-affine map of $\IF_q^d$}. An $\IF_q$-affine map $\lambda(M,v)$ is a permutation of $\IF_q^d$ if and only if $M$ is invertible, and the $\IF_q$-affine maps that are permutations form a permutation group on $\IF_q^d$ denoted by $\AGL_d(q)$.
\item We denote by $\CGL_d(q)$ the set of all matrices $A\in\GL_d(q)$ that do not have $-1$ as an eigenvalue (equivalently, that represent an $\IF_q$-linear complete mapping of $\IF_q^d$).
\item We set $\ACGL_d(q):=\{\lambda(A,v):A\in\CGL_d(q),v\in\IF_q^d\}\subseteq\AGL_d(q)$.
\end{enumerate}
\item Let $p$ be a prime, and let $d$ and $\ell$ be positive integers. We set
\[
\Gamma(d,p,\ell):=
\begin{cases}
\CT(\ACGL_d(p)), & \text{if }\ell=1, \\
\CT(\AGL_d(p)), & \text{if }\ell\geq2\text{ and }(d,p)\not=(1,2),(1,3),(2,2), \\
\emptyset, & \text{if }\ell\geq2\text{ and }(d,p)=(1,2), \\
\{x_1^3,x_3\}, & \text{if }\ell\geq2\text{ and }(d,p)=(1,3), \\
\{x_1^4,x_2^2,x_1x_3\}, & \text{if }\ell\geq2\text{ and }(d,p)=(2,2).
\end{cases}
\]
\end{enumerate}
\end{notation}
Section \ref{sec2} of this paper contains information on the cycle types of $\IF_q$-affine permutations of $\IF_q^d$. In principle, it is possible to compute the sets $\CT(\AGL_d(q))$ and $\CT(\ACGL_d(q))$ for each pair $(d,q)$ from this.

We are now ready to formulate the main result of this paper:
\begin{theorem}\label{mainTheo}
Let $p$ be a prime, and let $d$ and $t$ be positive integers. Assume that $x_1^{k_1}x_2^{k_2}\cdots x_{p^t}^{k_{p^t}}$ is the cycle type of a complete mapping of $\IF_p^t$. For $\ell=1,2,\ldots,p^t$ and $i=1,2,\ldots,k_{\ell}$, let $\gamma_{\ell,i}\in\Gamma(d,p,\ell)$. Then for every $d$-dimensional subspace $W$ of $\IF_p^{d+t}$, the $\IF_p$-vector space $\IF_p^{d+t}$ admits a $W$-coset-wise $\IF_p$-affine complete mapping of the cycle type
\[
\prod_{\ell=1}^{p^t}\prod_{i=1}^{k_{\ell}}{\BU_{\ell}(\gamma_{\ell,i})}.
\]
\end{theorem}
In fact, Theorem \ref{mainTheo} is a compact, inexplicit version of a more elaborate result, Theorem \ref{mainTheoExplicit}, which explicitly describes how to construct a $W$-coset-wise $\IF_p$-affine complete mapping of $\IF_p^{d+t}$ of the specified cycle type from a known complete mapping of $\IF_p^t$ of cycle type $x_1^{k_1}\cdots x_{p^t}^{k_{p^t}}$. For readers that are only interested in constructing permutations (but not necessarily complete mappings) of $\IF_q$ with a given cycle type, we note that there is a simpler analogue of Theorem \ref{mainTheo} (or, rather, its explicit version), where each of the two occurrences of \enquote{complete mapping} is replaced by \enquote{permutation}, and the set $\Gamma(d,p,\ell)$ may be replaced by its superset $\CT(\GL_d(p))$. For more details, see the end of Section \ref{sec4}.

As the formulation of Theorem \ref{mainTheo} is rather technical, it may be hard to gauge its power at a first glance. To demonstrate its usefulness, we derive the following interesting consequence:
\begin{corollary}\label{mainCor1}
Let $q=p^k$ be an odd prime power, and let $S$ be a Sylow $p$-subgroup of the symmetric group $\Sym(q)$. Then every cycle type of an element of $S$ is also the cycle type of a suitable complete mapping of $\IF_q$.
\end{corollary}
\begin{proof}[Proof of Corollary \ref{mainCor1} using Theorem \ref{mainTheo}]
We proceed by induction on $k$. For $k=1$, we have that $S$ is cyclic, generated by a $p$-cycle. Hence the only cycle types of elements of $S$ are $x_1^p$ and $x_p$. Since $\id_{\IF_p}$ and the function $\IF_p\rightarrow\IF_p$, $x\mapsto x+1$, are complete mappings of $\IF_p$ of cycle type $x_1^p$ and $x_p$ respectively, the claim is true in this case.

Now assume that $k>1$, and that the claim holds for $p^{k-1}$. Let $x_1^{a_0}x_p^{a_1}x_{p^2}^{a_2}\cdots x_{p^k}^{a_k}$ be the cycle type of an element of $S$. Then
\[
a_0\equiv\sum_{i=0}^k{a_ip^i}=p^k\equiv0\Mod{p}.
\]
The following is the cycle type of an element of a Sylow $p$-subgroup of $\Sym(p^{k-1})$:
\begin{equation}\label{cycleTypeEq}
x_1^{\frac{a_0}{p}+a_1}x_p^{a_2}x_{p^2}^{a_3}\cdots x_{p^{k-1}}^{a_k}.
\end{equation}
By the induction hypothesis, there is a complete mapping of $\IF_p^{k-1}$ of cycle type (\ref{cycleTypeEq}). We will apply Theorem \ref{mainTheo} with $d=1$, $t=k-1$, $x_1^{k_1}\cdots x_{p^t}^{k_{p^t}}$ equal to (\ref{cycleTypeEq}) and
\[
\gamma_{\ell,i}=
\begin{cases}
x_1^p, & \text{if }\ell=1\text{ and }i\in\{1,2,\ldots,\frac{a_0}{p}\}, \\
x_p, & \text{otherwise};
\end{cases}
\]
note that it is possible to choose the $\gamma_{\ell,i}$ like this because $\{x_1^p,x_p\}\subseteq\CT(\ACGL_1(p))\subseteq\Gamma(1,p,\ell)$ for all $\ell\geq1$. Our application of Theorem \ref{mainTheo} shows that
\[
\BU_1(x_1^p)^{\frac{a_0}{p}}\cdot\BU_1(x_p)^{a_1}\cdot\BU_p(x_p)^{a_2}\cdots\BU_{p^{k-1}}(x_p)^{a_k}=x_1^{a_0}x_p^{a_1}x_{p^2}^{a_2}\cdots x_{p^k}^{a_k}
\]
is the cycle type of a suitable complete mapping of $\IF_{p^k}$, as required.
\end{proof}

We note that Corollary \ref{mainCor1} is false for even $q$. In fact, if $q$ is even, then none of the cycle types of elements of a Sylow $2$-subgroup $S$ of $\Sym(q)$ can be the cycle type of a complete mapping of $\IF_q$. This is because all elements of $S$ have an even number of fixed points, whereas complete mappings of $\IF_q$ have precisely one fixed point (see the discussion before Definition \ref{cosetWiseDef}). Of course, this does not mean that Theorem \ref{mainTheo} is useless in characteristic $2$ -- one can still apply it to exhibit various cycle types of complete mappings. However, the situation is more complicated, because one cannot choose $d=1$ in Theorem \ref{mainTheo} if $p=2$ (as $\Gamma(1,2,\ell)=\emptyset$ for all $\ell\geq1$).

Corollary \ref{mainCor1} has the following interesting consequence, with which we conclude this introductory subsection:

\begin{corollary}\label{mainCor2}
Let $q$ be a prime power. The following are equivalent:
\begin{enumerate}
\item $\IF_q$ admits a complete mapping of cycle type $x_q$ (i.e., that permutes the elements of $\IF_q$ in one cycle).
\item $q$ is odd.
\end{enumerate}
\end{corollary}

\begin{proof}
Every Sylow $q$-subgroup of $\Sym(q)$ contains a $q$-cycle. Therefore, Corollary \ref{mainCor1} implies that $\IF_q$ has a complete mapping of cycle type $x_q$ if $q$ is odd. On the other hand, if $q$ is even, then as mentioned in the discussion before Definition \ref{cosetWiseDef}, a complete mapping of $\IF_q$ is also an orthomorphism of $\IF_q$ and thus has precisely one fixed point. In particular, the said complete mapping cannot be a $q$-cycle.
\end{proof}

In Section \ref{sec5}, we will explicitly construct, for each prime power $q$, a permutation of $\IF_q$ that is a $q$-cycle and, if $q$ is odd, is a complete mapping of $\IF_q$. This is an application of Theorem \ref{mainTheoExplicit}, the explicit version of Theorem \ref{mainTheo} mentioned above.

\subsection{Some notation used in this paper}\label{subsec1P2}

We denote by $\IN^+$ the set of positive integers. The symmetric group on a set $X$ is denoted by $\Sym(X)$. As is customary in group theory, all group actions in this paper are on the right, and we use the notations $f(x)$ and $x^f$ to denote the value of $x$ under the function $f$ interchangeably. In particular, we consider $\GL_d(q)$ as a group of matrices acting on \emph{row} vectors from $\IF_q^d$ through multiplication on the right ($v^M:=v\cdot M$ for $v\in\IF_q^d$ and $M\in\GL_d(q)$). As a consequence, our companion matrix forms (see formula (\ref{compMatrixEq})) are the transposes of the companion matrices used by Fripertinger \cite{Fri97a} (who worked with the left action of $\GL_d(q)$ on column vectors instead).

\subsection{Overview of this paper}\label{subsec1P3}

In Section \ref{sec2}, we recall the computation of cycle types of affine permutations of finite vector spaces, a problem studied and solved by Fripertinger in \cite{Fri97a}. This is useful for readers who want to use the theorem to construct explicit examples of complete mappings with a prescribed cycle type, as we do in Section \ref{sec5}.

Section \ref{sec3} is concerned with the problem of determining the product sets
\[
\CGL_d(q)^{(\ell)}=\{A_1\cdots A_{\ell}: A_i\in\CGL_d(q)\}
\]
for each triple $(d,q,\ell)$, see Proposition \ref{cglProp}. For us, this is an auxiliary problem for proving Theorem \ref{mainTheo}, but it is also interesting in its own right and has some connections with a group-theoretic problem recently studied by Larsen, Shalev and Tiep in \cite{LST20a}, see the end of Section \ref{sec3}.

In Section \ref{sec4}, we formulate and prove Theorem \ref{mainTheoExplicit}, the above-mentioned stronger (explicit) version of Theorem \ref{mainTheo}, and in Section \ref{sec5}, we explicitly construct one-cycle complete mappings of finite fields of odd characteristic as an application. Section \ref{sec6} concludes the paper with a discussion of a related, but harder problem which we aim to tackle in a follow-up paper.

\section{Cycle types of affine permutations of finite vector spaces}\label{sec2}

Let $q$ be a prime power, and let $V$ be a finite-dimensional $\IF_q$-vector space. Our goal in this section is to discuss how to compute the cycle type of a given affine permutation $\lambda(\alpha,v)$ of $V$, following Fripertinger \cite{Fri97a}, though we will use a slightly different presentation.

For this, we need the so-called primary rational canonical form of finite-dimensional vector space automorphisms, which we now briefly recall. For every field $K$, each endomorphism $\varphi$ of a finite-dimensional $K$-vector space $V$ induces a direct decomposition $V=\oplus_{i=1}^s{W_i}$ of $V$ into $\varphi$-invariant $K$-subspaces $W_i$ such that for $i=1,2,\ldots,s$, the restriction $\varphi_{\mid W_i}$ can be represented, with respect to a suitable basis of $W_i$, by $\Comp(Q_i^{e_i})$ where $Q_i\in K[X]$ is a monic irreducible polynomial, $e_i$ is a positive integer, and $\Comp(P)$ denotes the so-called \emph{companion matrix of the polynomial $P=X^d+a_{d-1}X^{d-1}+\cdots+a_1X+a_0\in K[X]$}, which is the following $(d\times d)$-matrix over $K$:
\begin{equation}\label{compMatrixEq}
\begin{pmatrix}
0 & 1 & 0 & \cdots & 0 & 0 \\
0 & 0 & 1 & \cdots & 0 & 0 \\
\vdots & \vdots & \vdots & \cdots & \vdots & \vdots \\
0 & 0 & 0 & \cdots & 1 & 0 \\
0 & 0 & 0 & \cdots & 0 & 1 \\
-a_0 & -a_1 & -a_2 & \cdots & -a_{d-2} & -a_{d-1}\end{pmatrix}.
\end{equation}
A direct decomposition $V=\oplus_{i=1}^s{W_i}$ of $V$ as described in the last sentence will be referred to as a \emph{$\varphi$-block subspace decomposition of $V$}. Observe that $\Comp(P)$ is the matrix representing the multiplication
\[
R+(P)\mapsto RX+(P)
\]
by $X$ on the quotient algebra $K[X]/(P)$ with respect to the basis $(1+(P),X+(P),X^2+(P),\ldots,X^{d-1}+(P))$. While the $\varphi$-block subspace decomposition $V=\oplus_{i=1}^s{W_i}$ is not necessarily unique, the multiset $\{Q_i^{e_i}:i=1,2,\ldots,s\}$ of powers of monic irreducible polynomials associated with the decomposition is uniquely determined by $\varphi$. The corresponding companion matrices $\Comp(Q_i^{e_i})$ are called the \emph{primary rational canonical blocks of $\varphi$}, and any block diagonal matrix that has these companion matrices as its diagonal blocks is called a \emph{primary rational canonical form of $\varphi$}. Observing that for each monic polynomial $P\in K[X]$, both the characteristic and the minimal polynomial of $\Comp(P)$ is $P$ itself, we conclude that the characteristic polynomial of $\varphi$ is $\prod_{i=1}^s{Q_i^{e_i}}$, whereas its minimal polynomial is $\lcm(Q_i^{e_i}: i=1,2,\ldots,s)$.

It follows that $\varphi$ is an automorphism of $V$ if and only if $Q_i\not=X$ for all $i$, and that $\varphi$ is a complete automorphism of $V$ if and only if $Q_i\not=X,X+1$ for all $i$. Assume henceforth that $\varphi=\alpha$ is an automorphism of $V$. Moreover, let $v=\sum_{i=1}^s{v_i}\in V$, with $v_i\in W_i$ for all $i$. Using that
\[
\CT(\lambda(\alpha,v))=\divideontimes_{i=1}^s{\CT(\lambda(\alpha_{\mid W_i},v_i))},
\]
where $\divideontimes$ is as in \cite[Definition 2.2]{WX93a}, we see that in order to understand the cycle types of (complete) affine permutations of $V$, it suffices to understand the cycle types of affine permutations $\lambda(\gamma,w)$ whose automorphism part $\gamma$ can be represented by a companion matrix of the form $\Comp(Q^e)$ for some positive integer $e$ and some monic irreducible polynomial $Q\in K[X]$ such that $Q\not=X$ (and $Q\not=X+1$ for the complete case). This is also what Fripertinger did in \cite{Fri97a}, though in case $e>1$, he replaced $\Comp(Q^e)$ by a certain similar, so-called \emph{hypercompanion matrix}. He obtained essentially the following result, formulated in terms of polynomial quotient algebras:

\begin{proposition}\label{fripertingerProp}
Let $q>1$ be a power of a prime $p$, let $Q,U\in\IF_q[X]$ with $Q\not=X$ monic irreducible, and let $e$ be a positive integer. Consider the affine permutation
\[
\lambda(X,U): R+(Q^e)\mapsto RX+U+(Q^e)
\]
of $\IF_q[X]/(Q^e)$.
\begin{enumerate}
\item If $Q\not=X-1$, then $\lambda(X,U)$ has the following cycle count (independently of $U$):
\begin{itemize}
\item $1$ fixed point;
\item $\frac{q^{\deg{Q}}-1}{\ord(Q)}$ cycles of length $\ord(Q)$;
\item for each $a=1,2,\ldots,\lceil\log_p(e)\rceil-1$: $\frac{q^{p^{a-1}\deg{Q}}(q^{\deg{Q}p^{a-1}(p-1)}-1)}{p^a\ord(Q)}$ cycles of length $\ord(Q)p^a$; and
\item $\frac{q^{p^{\lceil\log_p(e)\rceil}\deg{Q}}(q^{\deg{Q}(e-p^{\lceil\log_p(e)\rceil-1})}-1)}{p^{\lceil\log_p(e)\rceil}\ord(Q)}$ cycles of length $\ord(Q)p^{\lceil\log_p(e)\rceil}$.
\end{itemize}
\item If $Q=X-1$ and $U+(Q^e)$ is a non-unit in $\IF_q[X]/(Q^e)$, then $\lambda(X,U)$ has the following cycle count:
\begin{itemize}
\item $q$ fixed points;
\item for each $a=1,2,\ldots,\lceil\log_p(e)\rceil-1$: $\frac{q^{p^{a-1}}(q^{p^{a-1}(p-1)}-1)}{p^a}$ cycles of length $p^a$; and
\item $\frac{q^{p^{\lceil\log_p(e)\rceil-1}}(q^{e-p^{\lceil\log_p(e)\rceil-1}}-1)}{p^{\lceil\log_p(e)\rceil}}$ cycles of length $p^{\lceil\log_p(e)\rceil}$.
\end{itemize}
\item If $Q=X-1$, if $U+(Q^e)$ is a unit in $\IF_q[X]/(Q^e)$, and if $e>1$ is not a power of $p$, then $\lambda(X,U)$ has $\frac{q^e}{p^{\lceil\log_p(e)\rceil}}$ cycles, all of length $p^{\lceil\log_p(e)\rceil}$.
\item If $Q=X-1$, if $U+(Q^e)$ is a unit in $\IF_q[X]/(Q^e)$, and if $e$ is a power of $p$ (including the case $e=1$), then $\lambda(X,U)$ has $\frac{q^e}{pe}$ cycles, all of length $pe$.
\end{enumerate}
\end{proposition}

Because our presentation differs slightly from the one of Fripertinger (we chose to avoid using hypercompanion matrices for greater uniformity), we give a self-contained proof of Proposition \ref{fripertingerProp} for the reader's convenience:

\begin{proof}[Proof of Proposition \ref{fripertingerProp}]
For each positive integer $\ell$, we have
\[
\lambda(X,U)^{\ell}(R+(Q^e))=RX^{\ell}+U(1+X+X^2+\cdots+X^{\ell-1})+(Q^e)
\]
for all $R\in\IF_q[X]$. Therefore, the number of solutions $R$ modulo $Q^e$ of the congruence
\[
RX^{\ell}+U(1+X+X^2+\cdots+X^{\ell-1})\equiv X\Mod{Q^e}
\]
equals the number of points in $\IF_q[X]/(Q^e)$ that lie on a cycle of $\lambda(X,U)$ whose length divides $\ell$. This congruence is equivalent to
\begin{equation}\label{congruenceEq}
R(X-1)(1+X+\cdots+X^{\ell-1})=R(X^{\ell}-1)\equiv -U(1+X+\cdots+X^{\ell-1})\Mod{Q^e}.
\end{equation}
We now make a case distinction according to the four statements we need to prove:
\begin{enumerate}
\item Case: $Q\not=X-1$. Since $X-1$ is a unit modulo $Q^e$, we conclude that the precise number of solutions modulo $Q^e$ of congruence (\ref{congruenceEq}) is
\[
q^{\deg{Q}\cdot\min(e,\nu_Q(X^{\ell}-1))}.
\]
Observing that
\[
\nu_Q(X^k-1)=\begin{cases}0, & \text{if }\ord(Q)\nmid k, \\ p^{\nu_p(\frac{k}{\ord(Q)})}, & \text{if }\ord(Q)\mid k\end{cases}
\]
we see that the cycle lengths of $\lambda(X,U)$ are $1$ and the numbers of the form $\ord(Q)p^a$ for some $a=0,1,\ldots,\lceil\log_p(e)\rceil$ (those are the values of $\ell$ for which the number of solutions of congruence (\ref{congruenceEq}) is strictly larger than the number of solutions for any proper divisor of $\ell$). Moreover, the number of fixed points of $\lambda(X,U)$ is precisely
\[
q^{\deg{Q}\cdot\min(e,\nu_Q(X-1))}=q^{\deg{Q}\cdot\min(e,0)}=q^0=1,
\]
whereas for $a=0,1,\ldots,\lceil\log_p(e)\rceil$, the number of solutions of congruence (\ref{congruenceEq}) for $\ell=\ord(Q)p^a$ is precisely
\[
q^{\deg{Q}\cdot\min(e,p^a)}.
\]
An inclusion-exclusion counting argument now confirms the asserted cycle count of $\lambda(X,U)$.
\item Case: $Q=X-1$, and $U+(Q^e)$ is a non-unit in $\IF_q[X]/(Q^e)$ (i.e., $Q\mid U$). Then the $Q$-adic valuation of the right-hand side $-U(1+X+X^2+\cdots+X^{\ell-1})$ of congruence (\ref{congruenceEq}) is at least $1+\nu_Q(1+X+X^2+\cdots+X^{\ell-1})$, which is the precise $Q$-adic valuation of the coefficient $X^{\ell}-1$ of the left-hand side of that congruence. It follows that congruence (\ref{congruenceEq}) is solvable for all $\ell$ and, more precisely, its number of solutions modulo $Q^e$ is
\[
q^{\min(e,\nu_Q(X^{\ell}-1))}.
\]
Since $Q=X-1$, we have that
\[
\nu_Q(X^k-1)=p^{\nu_p(k)}
\]
and conclude that the cycle lengths of $\lambda(X,U)$ are just the numbers of the form $p^a$ with $a=0,1,\ldots,\lceil\log_p(e)\rceil$, with precisely $q^{\min(e,p^a)}$ points lying on a cycle whose length divides $p^a$. As in Case (1), an inclusion-exclusion counting argument now yields the asserted cycle count of $\lambda(X,U)$.
\item Case: $Q=X-1$, $U+(Q^e)$ is a unit in $\IF_q[X]/(Q^e)$ (i.e., $Q\nmid U$), and $e>1$ is not a power of $p$. For $\ell=p^{\lceil\log_p(e)\rceil}$, congruence (\ref{congruenceEq}) becomes the universally solvable
\[
0\equiv R(X-1)^{p^{\lceil\log_p(e)\rceil}}\equiv -U(X-1)^{p^{\lceil\log_p(e)\rceil}-1}\equiv0\Mod{(X-1)^e},
\]
so all cycles of $\lambda(X,U)$ have length dividing $p^{\lceil\log_p(e)\rceil}$. On the other hand, if $a\in\{0,1,\ldots,\lceil\log_p(e)\rceil-1\}$ and $\ell=p^a$, then congruence (\ref{congruenceEq}) has no solutions, because the $Q$-adic valuation $p^a$ of the left-hand side coefficient $(T^{p^a}-1)=(T-1)^{p^a}$ is strictly larger than $p^a-1$, the $Q$-adic valuation of the right-hand side $-U(T-1)^{p^a-1}$. It follows that all cycles of $\lambda(X,U)$ have the length $p^{\lceil\log_p(e)\rceil}$, as required.
\item Case: $Q=X-1$, $U+(Q^e)$ is a unit in $\IF_q[X]/(Q^e)$ (i.e., $Q\nmid U$), and $e$ is a power of $p$ (possibly $e=1$). Then an argument analogous to the one of Case (3) shows that all cycles of $\lambda(X,U)$ have length $p^{\lceil\log_p(e)\rceil+1}=pe$.
\end{enumerate}
\end{proof}

The following notation will come in handy in the next example:

\begin{notation}\label{gammaNot}
Let $q$ be a prime power, and let $M$ be an invertible $(n\times n)$-matrix over $\IF_q$.
\begin{enumerate}
\item We denote by $\Gamma(M)=\Gamma_{\IF_q}(M)$ the set of all cycle types of affine permutations of $\IF_q^n$ of the form $\lambda(M,v)$, with $v\in\IF_q^n$.
\item For a non-constant monic polynomial $P\in\IF_q[X]$ with $P(0)\not=0$, we set $\Gamma(P):=\Gamma(\Comp(P))$.
\end{enumerate}
\end{notation}

\begin{example}\label{cycleTypeEx}
Let $q=3$ and $V=\IF_3^7$. Consider the $\IF_3$-endomorphism $\alpha$ of $V$ whose standard matrix is the block diagonal matrix with blocks
\[
\Comp((X-1)^2),\Comp((X-1)^3),\text{ and }\Comp(X^2+X+2).
\]
This diagonal matrix is a primary rational canonical form of $\alpha$. We list the possible cycle types of (complete) affine permutations of $V$ of the form $\lambda(\alpha,v)$ for some $v\in V$.

By Proposition \ref{fripertingerProp}, we have
\[
\Gamma((X-1)^2)=\{x_1^3x_3^2,x_3^3\}, \Gamma((X-1)^3)=\{x_1^3x_3^8,x_9^3\},\text{ and }\Gamma(X^2+X+2)=\{x_1x_8\}.
\]
This yields the following possibilities for $\CT(\lambda(\alpha,v))$ (note that the second and fourth of them are equal), which we computed via an implementation of Wei-Xu's product $\divideontimes$ in GAP \cite{GAP4}:
\begin{itemize}
\item $(x_1^3x_3^2)\divideontimes(x_1^3x_3^8)\divideontimes(x_1x_8)=x_1^9x_3^{78}x_8^9x_{24}^{78}$,
\item $(x_1^3x_3^2)\divideontimes x_9^3\divideontimes(x_1x_8)=x_9^{27}x_{72}^{27}$,
\item $x_3^3\divideontimes(x_1^3x_3^8)\divideontimes(x_1x_8)=x_3^{81}x_{24}^{81}$,
\item $x_3^3\divideontimes x_9^3\divideontimes(x_1x_8)=x_9^{27}x_{72}^{27}$.
\end{itemize}
\end{example}

\section{Products of complete vector space automorphisms}\label{sec3}

In order to prove Theorem \ref{mainTheo}, we will need to understand which elements of $\GL_d(q)$ can be written as products of matrices in $\CGL_d(q)$ with a given number $\ell$ of factors. The following proposition solves this problem:

\begin{proposition}\label{cglProp}
Let $d$ and $\ell$ be positive integers, and let $q$ be a prime power. Set
\[
\CGL_d(q)^{(\ell)}:=\{A_1\cdots A_{\ell}: A_i\in\CGL_d(q)\text{ for }i=1,2,\ldots,\ell\}\subseteq\GL_d(q).
\]
Then
\[
\CGL_d(q)^{(\ell)}=
\begin{cases}
\CGL_d(q), & \text{if }\ell=1, \\
\GL_d(q), & \text{if }\ell\geq2\text{ and }(d,q)\not=(1,2),(1,3),(2,2), \\
\emptyset, & \text{if }\ell\geq2\text{ and }(d,q)=(1,2), \\
\{(1)\}, & \text{if }\ell\geq2\text{ and }(d,q)=(1,3), \\
\left\{\begin{pmatrix}1 & 0 \\ 0 & 1\end{pmatrix},\begin{pmatrix}0 & 1 \\ 1 & 1\end{pmatrix},\begin{pmatrix}1 & 1 \\ 1 & 0\end{pmatrix}\right\}, & \text{if }\ell\geq2\text{ and }(d,q)=(2,2).
\end{cases}
\]
\end{proposition}

\begin{proof}
It is clear by definition that $\CGL_d(q)^{(1)}=\CGL_d(q)$, so we may assume that $\ell\geq2$. First, we discuss the three exceptional cases $(d,q)=(1,2),(1,3),(2,2)$. For $(d,q)=(1,2)$, we have $\CGL_d(q)=\CGL_1(2)=\emptyset$, because $\IF_2$ has no complete mappings. This implies $\CGL_1(2)^{(\ell)}=\emptyset$ for all $\ell\in\IN^+$, in particular for $\ell\geq2$. For $(d,q)=(1,3),(2,2)$, one can easily check the asserted equalities for $\ell=2$. In particular, we find that in those cases, $\CGL_d(q)^{(2)}$ is a subgroup of $\GL_d(q)$ containing $\CGL_d(q)$, which implies that $\CGL_d(q)^{(\ell)}=\CGL_d(q)^{(2)}$ for all $\ell\geq2$, as required.

Now we assume in addition to $\ell\geq2$ that $(d,q)\not=(1,2),(1,3),(2,2)$. Under these assumptions, we need to show that $\CGL_d(q)^{(\ell)}=\GL_d(q)$. As in the last paragraph, if we can only show this assertion for $\ell=2$, it is clear that it holds for all $\ell\geq2$. So we will focus on proving that $\CGL_d(q)^{(2)}=\GL_d(q)$. We distinguish a few cases.

First, assume that $d=1$ and $q>3$. Then $\CGL_d(q)=\CGL_1(q)=\{(a): a\in\IF_q^{\ast},a\not=-1\}$. Because $q>3$, we have $|\IF_q^{\ast}|=q-1\geq3$. Hence, for each $a\in\IF_q^{\ast}$, we can choose an element $b\in\IF_q^{\ast}\setminus\{-1,-a\}$. Then $a=b\cdot\frac{a}{b}$ is a representation of $a$ as a product of two elements of $\IF_q^{\ast}$ that are distinct from $-1$. This shows that $\CGL_1(q)^{(2)}=\GL_1(q)$, as required.

Next, assume that $d>1$ and $q>2$. For each one-dimensional $\IF_q$-subspace $L$ of $\IF_q^d$, set
\[
G_L:=\{A\in\GL_d(q): L\text{ is contained in the }(-1)\text{-eigenspace of } A\}.
\]
Observe that for each $L$, we have $|G_L|=\prod_{i=1}^{d-1}{(q^d-q^i)}$, that $\GL_d(q)\setminus\CGL_d(q)\subseteq\bigcup_L{G_L}$, and that $-I_d\in G_L$ for all $L$. Since $\IF_q^d$ has precisely $\frac{q^d-1}{q-1}$ one-dimensional $\IF_q$-subspaces $L$, this implies that
\begin{align*}
&|\GL_d(q)\setminus\CGL_d(q)|\leq\frac{q^d-1}{q-1}\cdot\prod_{i=1}^{d-1}{(q^d-q^i)}-\left(\frac{q^d-1}{q-1}-1\right)= \\
&\frac{1}{q-1}|\GL_d(q)|-\left(\frac{q^d-1}{q-1}-1\right)<\frac{1}{q-1}|\GL_d(q)|\leq\frac{1}{2}|\GL_d(q)|.
\end{align*}
Hence, for each given $A\in\GL_d(q)$, each of the two sets $\CGL_d(q)$ and $A\cdot\CGL_d(q)^{-1}$ has size larger than $\frac{1}{2}|\GL_d(q)|$, whence $|\CGL_d(q)\cap(A\cdot\CGL_d(q)^{-1})|>0$. That is, there are $C_1,C_2\in\CGL_d(q)$ such that $C_1=A\cdot C_2^{-1}$ or, equivalently, $A=C_1C_2\in\CGL_d(q)^{(2)}$. This shows that $\CGL_d(q)^{(2)}=\GL_d(q)$, as required.

This leaves us with the assumptions $d>1$ and $q=2$. Since $q$ is even, we have $-1=1$ in $\IF_q$, whence $\CGL_d(q)=\CGL_d(2)$ consists of exactly those matrices $A\in\GL_d(2)$ that are \emph{fixed-point-free} (henceforth abbreviated to f.p.f.), which is defined to mean that those matrices have no \emph{nonzero} fixed points. We use this to rewrite the assertion that $\CGL_d(2)^{(2)}=\GL_d(2)$ into an easier to handle equivalent form as follows: For a given matrix $A\in\GL_d(2)$, there exist f.p.f.~matrices $C_1,C_2\in\GL_d(2)$ such that $A=C_1C_2$ if and only if there exists an f.p.f.~matrix $C\in\GL_d(2)$ such that $AC^{-1}$ is f.p.f.~as well. Now, $AC^{-1}$ is f.p.f.~if and only if for each vector $v\in\IF_2^d$, the condition $vAC^{-1}=v$, which is equivalent to $vA=vC$, implies $v=\vec{0}$. That is, $AC^{-1}$ is f.p.f.~if and only if $\ker(A-C)=\ker(A+C)=\{\vec{0}\}$, i.e., if and only if $A+C\in\GL_d(2)$.

In view of this, we are done once we have shown the following claim:
\begin{equation}\label{claimEq}
\text{For each }A\in\GL_d(2)\text{, there is an f.p.f. }C\in\GL_d(2)\text{ such that }A+C\in\GL_d(2).
\end{equation}
Observe that since the set $\CGL_d(2)$ of f.p.f.~invertible $(d\times d)$-matrices over $\IF_2$ is closed under conjugation by matrices in $\GL_d(2)$, if claim (\ref{claimEq}) holds for a given $A\in\GL_d(2)$, it also holds for every matrix in the conjugacy class $A^{\GL_d(2)}$. Therefore, it suffices to prove claim (\ref{claimEq}) for only one matrix $A$ per $\GL_d(2)$-conjugacy class.

We prove claim (\ref{claimEq}) by induction on $d\geq3$. One can check with GAP \cite{GAP4} that the claim holds for $3\leq d\leq 6$. We may thus assume that $d\geq7$ and that the claim holds in dimensions $3,4,\ldots,d-1$. Let $A\in\GL_d(2)$ be arbitrary but fixed, and let $\IF_2^d=\bigoplus_{i=1}^s{V_i}$ be an $A$-block subspace decomposition of $\IF_2^d$, with blocks $\Comp(Q_i^{e_i})$.

If it is possible to partition the multiset of primary rational canonical blocks of $A$ into two submultisets such that the sum of the block dimensions of each submultiset lies in $\{3,4,\ldots,d-3\}$, then we are done by the induction hypothesis. Indeed, we then have that $A$ can be written as a block diagonal matrix
\[
\begin{pmatrix}
A_1 & 0 \\
0 & A_2
\end{pmatrix}
\]
such that each diagonal block $A_i$ has dimension $d_i\in\{3,4,\ldots,d-3\}$. The induction hypothesis yields that there exist matrices $C_i\in\CGL_{d_i}(2)$ for $i=1,2$ such that $A_i+C_i\in\GL_{d_i}(2)$. Setting
\[
C:=
\begin{pmatrix}
C_1 & 0 \\
0 & C_2
\end{pmatrix},
\]
we find that $C\in\CGL_d(2)$ and $A+C\in\GL_d(2)$, as required.

Therefore, we may assume that a partition of the multiset $M$ of primary rational canonical blocks of $A$ as described above is not possible. This leaves the following possibilities for the multiset $M'$ of the dimensions of the primary rational canonical blocks of $A$: $\{d\}$, $\{1,d-1\}$, $\{1,1,d-2\}$, and $\{2,d-2\}$. Indeed, $A$ cannot have any primary rational canonical block of a dimension in $\{3,4,\ldots,d-3\}$ -- otherwise, consider the bipartition of $M$ where one partition class is just the singleton consisting of a block of $A$ with dimension in $\{3,4,\ldots,d-3\}$. Therefore, the possible elements of $M'$ are $1$, $2$, $d-2$, $d-1$ and $d$. Since $\sum{M'}=d$, we find that $M'=\{d\}$ if $d\in M'$. Now assume that $d\notin M'$. Since $d\geq7$, we have $d-1>d-2>\frac{1}{2}d$, so $M'$ can only contain one of $d-1$ and $d-2$, and only with multiplicity $1$. In particular, $M'$ must contain at least one of the numbers $1$ and $2$. But since $3\leq d-3$, we have that $M'$ cannot contain both $1$ and $2$ -- otherwise, consider the bipartition of $M$ where one partition class consists of precisely one block each of dimensions $1$ and $2$. Similarly, using that $3\leq d-3$ resp.~$4\leq d-3$, we see that the multiplicity of $1$ resp.~$2$ in $M'$ is at most $2$ resp.~$1$. Hence, $M'$ must contain precisely one of $d-1$ or $d-2$, and it does so with multiplicity $1$. Since all other elements of $M'$ are equal to $1$ or $2$, it now follows that $M'$ is one of the remaining three multisets listed at the beginning of this paragraph.

We now go through the four possibilities for $M'$ in a case distinction.
\begin{enumerate}
\item Case: $M'=\{d\}$. Let $\Comp(Q^e)$ be the unique primary rational canonical block of $A$. First, assume that $Q\not=X+1$. Then $\ord(Q)$ is an odd number greater than $1$, and since $A$ is f.p.f.~and has no $2$-cycles according to Proposition \ref{fripertingerProp}, we find that both $A^{-1}$ and $A^2$ are f.p.f., whence $C:=A^{-1}$ is an f.p.f.~invertible matrix such that $A+C=A+A^{-1}=A^{-1}(A^2+I_d)\in\GL_d(2)$, as required.

Now assume that $Q=X+1$. Then we may assume without loss of generality that $A$ is the hypercompanion matrix
\[
\begin{pmatrix}
1 & 1 & 0 & 0 & \cdots & 0 & 0 \\
0 & 1 & 1 & 0 & \cdots & 0 & 0 \\
\vdots & \vdots & \vdots & \vdots & \cdots & \vdots & \vdots \\
0 & 0 & 0 & 0 & \cdots & 1 & 1 \\
0 & 0 & 0 & 0 & \cdots & 0 & 1
\end{pmatrix}.
\]
For $C$, we make the ansatz of choosing it as a companion matrix
\begin{equation}\label{cAnsatzEq}
C=\Comp(X^d+b_{d-1}X^{d-1}+\cdots+b_1X+b_0)=
\begin{pmatrix}
0 & 1 & 0 & 0 & \cdots & 0 & 0 \\
0 & 0 & 1 & 0 & \cdots & 0 & 0 \\
\vdots & \vdots & \vdots & \vdots & \cdots & \vdots & \vdots \\
0 & 0 & 0 & 0 & \cdots & 0 & 1 \\
b_0 & b_1 & b_2 & b_3 & \cdots & b_{d-2} & b_{d-1}
\end{pmatrix}.
\end{equation}
Note that $C\in\CGL_d(2)$ if and only if $b_0=1$ and $b_1+\cdots+b_{d-1}=1$. Indeed, $C\in\GL_d(2)=\SL_d(2)$ if and only if $b_0=\det{C}=1$, and under this assumption, $C$ is f.p.f.~if and only if the characteristic polynomial $\chi_C=1+b_1X+b_2X^2+\cdots+b_{d-1}X^{d-1}+X^d$ is not divisible by $X+1$ (i.e., does not have $1$ as a root), which is equivalent to $b_1+b_2+\cdots+b_{d-1}=1$.

Now, observe that
\[
A+C=
\begin{pmatrix}
1 & 0 & 0 & \cdots & 0 & 0 \\
0 & 1 & 0 & \cdots & 0 & 0 \\
\vdots & \vdots & \vdots & \cdots & \vdots & \vdots \\
0 & 0 & 0 & \cdots & 1 & 0 \\
b_0 & b_1 & b_2 & \cdots & b_{d-2} & b_{d-1}+1
\end{pmatrix},
\]
and $A+C\in\GL_d(2)$ if and only if the rows of $A+C$ span $\IF_2^d$, which is the case if and only if $b_{d-1}=0$. Hence, if we choose $b_0:=b_1:=1$ and $b_2:=b_3:=\cdots:=b_{d-1}:=0$ in formula (\ref{cAnsatzEq}), the resulting matrix $C$ will be invertible and f.p.f.~and satisfy $A+C\in\GL_d(2)$, as required.
\item Case: $M'=\{1,d-1\}$. Let $\Comp(Q^e)$ be the primary rational canonical block of $A$ of dimension $d-1$. First, assume that $Q\not=X+1$. Then we assume without loss of generality that $A$ is its own primary rational canonical form, i.e.,
\[
A=
\begin{pmatrix}
0 & 1 & 0 & 0 & \cdots & 0 & 0 \\
0 & 0 & 1 & 0 & \cdots & 0 & 0 \\
\vdots & \vdots & \vdots & \vdots & \cdots & \vdots & \vdots \\
0 & 0 & 0 & 0 & \cdots & 1 & 0 \\
a_0 & a_1 & a_2 & a_3 & \cdots & a_{d-2} & 0 \\
0 & 0 & 0 & 0 & \cdots & 0 & 1
\end{pmatrix}
\]
where $Q^e=a_0+a_1X+a_2X^2+\cdots+a_{d-2}X^{d-2}+X^{d-1}$. Note that since $X+1\nmid Q^e$, the sum of all coefficients of $Q^e$ is $1$, i.e., $a_0+a_1+\cdots+a_{d-2}=0$.

For $C$, we choose the ansatz
\begin{align*}
C&=
\begin{pmatrix}
0 & 0 & \cdots & 0 & 1 \\
0 & 0 & \cdots & 1 & 0 \\
\vdots & \vdots & \vdots & \vdots & \vdots \\
0 & 1 & \cdots & 0 & 0 \\
1 & 0 & \cdots & 0 & 0
\end{pmatrix}
\cdot
\begin{pmatrix}
0 & 1 & 0 & \cdots & 0 \\
0 & 0 & 1 & \cdots & 0 \\
\vdots & \vdots & \vdots & \cdots & \vdots \\
0 & 0 & 0 & \cdots & 1 \\
1 & b_1 & b_2 & \cdots & b_{d-1}
\end{pmatrix}
\cdot
\begin{pmatrix}
0 & 0 & \cdots & 0 & 1 \\
0 & 0 & \cdots & 1 & 0 \\
\vdots & \vdots & \vdots & \vdots & \vdots \\
0 & 1 & \cdots & 0 & 0 \\
1 & 0 & \cdots & 0 & 0
\end{pmatrix} \\
&=
\begin{pmatrix}
b_{d-1} & b_{d-2} & \cdots & b_1 & 1 \\
1 & 0 & \cdots & 0 & 0 \\
0 & 1 & \cdots & 0 & 0 \\
\vdots & \vdots & \cdots & \vdots & \vdots \\
0 & 0 & \cdots & 1 & 0
\end{pmatrix}
\end{align*}
with $b_1+\cdots+b_{d-1}=1$ (see the discussion after formula (\ref{cAnsatzEq}) in the previous case). Then
\[
A+C=
\begin{pmatrix}
b_{d-1} & b_{d-2}+1 & b_{d-3} & b_{d-4} & b_{d-5} & \cdots & b_3 & b_2 & b_1 & 1  \\
1 & 0 & 1 & 0 & 0 & \cdots & 0 & 0 & 0 & 0 \\
0 & 1 & 0 & 1 & 0 & \cdots & 0 & 0 & 0 & 0 \\
\vdots & \vdots & \vdots & \vdots & \vdots & \cdots & \vdots & \vdots & \vdots & \vdots \\
0 & 0 & 0 & 0 & 0 & \cdots & 1 & 0 & 1 & 0 \\
a_0 & a_1 & a_2 & a_3 & a_4 & \cdots & a_{d-4} & a_{d-3}+1 & a_{d-2} & 0 \\
0 & 0 & 0 & 0 & 0 & \cdots & 0 & 0 & 1 & 1
\end{pmatrix}.
\]
For $i=1,2,\ldots,d-1$, set
\[
b_i:=
\begin{cases}
a_{d-1-i}, & \text{if }i\not=d-2, \\
a_1+1, & \text{if }i=d-2.
\end{cases}
\]
Note that with this choice of $b_i$, we have $b_1+\cdots+b_{d-1}=a_0+\cdots+a_{d-2}+1=1$, so $C\in\CGL_d(2)$. Moreover, the difference between the first and penultimate rows of $A+C$ is the vector
\[
\begin{pmatrix}
0 & \cdots & 0 & 1 & 0 & 1
\end{pmatrix}.
\]
It follows that
\begin{equation}\label{spanEq}
\RowSpace_{\IF_q}
{
\begin{pmatrix}
1 & 0 & 1 & 0 & \cdots & 0 & 0 & 0 \\
0 & 1 & 0 & 1 & \cdots & 0 & 0 & 0 \\
\vdots & \vdots & \vdots & \vdots & \cdots & \vdots & \vdots & \vdots \\
0 & 0 & 0 & 0 & \cdots & 1 & 0 & 1 \\
0 & 0 & 0 & 0 & \cdots & 0 & 1 & 1
\end{pmatrix}
}\subseteq\im(A+C).
\end{equation}
We claim that the row space in formula (\ref{spanEq}) equals the hyperplane $H$ of $\IF_2^d$ that consists of those vectors whose entry sum is $0$. Indeed, a vector
\[
v=
\begin{pmatrix}
x_0 & x_1 & \cdots & x_{d-1}
\end{pmatrix}
\in\IF_2^d
\]
lies in this row space if and only if there are scalars $\lambda_0,\lambda_1,\ldots,\lambda_{d-2}\in\IF_2$ such that
\begin{align*}
&\begin{pmatrix}
x_0 & x_1 & \cdots & x_{d-1}
\end{pmatrix}
= \\
&\lambda_0
\begin{pmatrix}1 & 0 & 1 & 0 & \cdots & 0 & 0 & 0 \end{pmatrix}
+ \\
&\lambda_1
\begin{pmatrix}0 & 1 & 0 & 1 & \cdots & 0 & 0 & 0\end{pmatrix}
+ \\
&\cdots
+ \\
&\lambda_{d-3}
\begin{pmatrix}0 & 0 & 0 & 0 & \cdots & 1 & 0 & 1\end{pmatrix}
+ \\
&\lambda_{d-2}
\begin{pmatrix}0 & 0 & 0 & 0 & \cdots & 0 & 1 & 1\end{pmatrix},
\end{align*}
which translates to the equation system
\begin{align}\label{hEqSystem}
\notag\lambda_0&=x_0 \\
\notag\lambda_1&=x_1 \\
\notag\lambda_0+\lambda_2&=x_2 \\
\notag\lambda_1+\lambda_3&=x_3 \\
\notag&\vdots \\
\notag\lambda_{d-5}+\lambda_{d-3}&=x_{d-3} \\
\notag\lambda_{d-4}+\lambda_{d-2}&=x_{d-2} \\
\lambda_{d-3}+\lambda_{d-2}&=x_{d-1}
\end{align}
It is not hard to check that the equation system obtained from the system (\ref{hEqSystem}) by deleting the last equation, $\lambda_{d-3}+\lambda_{d-2}=x_{d-1}$, has a unique solution for every choice of $x_0,\ldots,x_{d-2}$, namely the one where for $i=0,1,\ldots,d-2$, one has
\[
\lambda_i=\sum_{j\leq i,2\mid i-j}{x_j};
\]
that is, $\lambda_0=x_0$, $\lambda_1=x_1$, $\lambda_2=x_0+x_2$, $\lambda_3=x_1+x_3$, $\lambda_4=x_0+x_2+x_4$, and so on. Therefore, the whole system (\ref{hEqSystem}) is solvable if and only if
\[
\sum_{j\leq d-3,2\mid d-3-j}{x_j}+\sum_{j\leq d-2,2\mid d-2-j}{x_j}=x_{d-1},
\]
which is equivalent to
\[
x_0+x_1+\cdots+x_{d-1}=0,
\]
i.e., to $v\in H$. Hence, the row space in formula (\ref{spanEq}) is indeed equal to $H$, as asserted. Now, since $\im(A+C)$ contains both $H$ and the vector
\[
\begin{pmatrix}a_0 & a_1 & \cdots & a_{d-4} & a_{d-3}+1 & a_{d-2} & 0\end{pmatrix}\notin H,
\]
we conclude that $\im(A+C)=\IF_2^d$, i.e., that $A+C\in\GL_d(2)$, as required.

Having dealt with the assumption $Q\not=X+1$, let us now assume that $Q=X+1$. Replacing the nontrivial primary Frobenius block $\Comp((X+1)^{d-1})$ of $A$ by its similar hypercompanion matrix, we may assume without loss of generality that
\[
A=
\begin{pmatrix}
1 & 1 & 0 & 0 & \cdots & 0 & 0 & 0 \\
0 & 1 & 1 & 0 & \cdots & 0 & 0 & 0 \\
\vdots & \vdots & \vdots & \vdots & \cdots & \vdots & \vdots & \vdots \\
0 & 0 & 0 & 0 & \cdots & 1 & 1 & 0 \\
0 & 0 & 0 & 0 & \cdots & 0 & 1 & 0 \\
0 & 0 & 0 & 0 & \cdots & 0 & 0 & 1
\end{pmatrix}.
\]
We make the ansatz
\[
C=\Comp(1+b_1X+b_2X^2+\cdots+b_{d-1}X^{d-1}+X^d)=
\begin{pmatrix}
0 & 1 & 0 & \cdots & 0 \\
0 & 0 & 1 & \cdots & 0 \\
\vdots & \vdots & \vdots & \cdots & \vdots \\
0 & 0 & 0 & \cdots & 1 \\
1 & b_1 & b_2 & \cdots & b_{d-1}
\end{pmatrix}
\]
with $b_1+\cdots+b_{d-1}=1$. Then
\[
A+C=
\begin{pmatrix}
1 & 0 & \cdots & 0 & 0 & 0 \\
0 & 1 & \cdots & 0 & 0 & 0 \\
\vdots & \vdots & \cdots & \vdots & \vdots & \vdots \\
0 & 0 & \cdots & 1 & 0 & 0 \\
0 & 0 & \cdots & 0 & 1 & 1 \\
1 & b_1 & \cdots & b_{d-3} & b_{d-2} & b_{d-1}+1
\end{pmatrix}.
\]
If we choose $b_1:=1$ and $b_2:=b_3:=\cdots:=b_{d-1}:=0$, then the rows of $A+C$ span $\IF_2^d$, whence $A+C\in\GL_d(2)$, as required.
\item Case: $M'=\{1,1,d-2\}$. Let $\Comp(Q^e)$ be the unique primary Frobenius block of $A$ of dimension $d-2$. Without loss of generality, we may assume that $A$ is its own primary rational canonical form. Hence, $A$ can be viewed as a block diagonal matrix with diagonal blocks $\Comp(Q^e)$ and $I_2$. By the induction hypothesis and the statement for $(d,q)=(2,2)$, we find that there are matrices $C_1\in\CGL_{d-2}(2)$ and $C_2\in\CGL_2(2)$ such that $\Comp(Q^e)+C_1\in\GL_{d-2}(2)$ and $I_2+C_2\in\GL_2(2)$. The block diagonal matrix $C$ with diagonal blocks $C_1$ and $C_2$ lies in $\CGL_d(2)$ and satisfies $A+C\in\GL_d(2)$, as required.
\item Case: $M'=\{2,d-2\}$. The unique primary Frobenius block $B$ of $A$ with dimension $2$ can be either $\Comp((X+1)^2)=\Comp(X^2+1)$ or $\Comp(X^2+X+1)$. If $B=\Comp(X^2+X+1)$, then $B\in\CGL_2(2)\subseteq\CGL_2(2)^{(2)}$, and we are done by an argument analogous to the one in the previous case. We may thus assume that
\[
B=\Comp(X^2+1)=\begin{pmatrix}0 & 1 \\ 1 & 0\end{pmatrix}\sim\begin{pmatrix}1 & 0 \\ 1 & 1\end{pmatrix}=:B'.
\]
Write the unique primary Frobenius block of $A$ of dimension $d-2$ as $\Comp(Q^e)$. First, assume that $Q\not=X+1$. We may assume without loss of generality that $A$ is the block diagonal matrix with diagonal blocks $\Comp(Q^e)$ and $B'$, i.e., that
\[
A=
\begin{pmatrix}
0 & 1 & 0 & \cdots & 0 & 0 & 0 \\
0 & 0 & 1 & \cdots & 0 & 0 & 0 \\
\vdots & \vdots & \vdots & \cdots & \vdots & \vdots & \vdots \\
0 & 0 & 0 & \cdots & 1 & 0 & 0 \\
a_0 & a_1 & a_2 & \cdots & a_{d-3} & 0 & 0 \\
0 & 0 & 0 & \cdots & 0 & 1 & 0 \\
0 & 0 & 0 & \cdots & 0 & 1 & 1
\end{pmatrix}
\]
where $Q^e=a_0+a_1X+\cdots+a_{d-3}X^{d-3}+X^{d-2}$, and we know that $a_0+\cdots+a_{d-3}=0$. As in the case \enquote{$M'=\{1,d-1\}$}, we make the ansatz
\[
C=
\begin{pmatrix}
b_{d-1} & b_{d-2} & \cdots & b_1 & 1 \\
1 & 0 & \cdots & 0 & 0 \\
0 & 1 & \cdots & 0 & 0 \\
\vdots & \vdots & \cdots & \vdots & \vdots \\
0 & 0 & \cdots & 1 & 0
\end{pmatrix}
\]
with $b_1+\cdots+b_{d-1}=1$. Then
\[
A+C=
\begin{pmatrix}
b_{d-1} & b_{d-2}+1 & b_{d-3} & b_{d-4} & \cdots & b_4 & b_3 & b_2 & b_1 & 1 \\
1 & 0 & 1 & 0 & \cdots & 0 & 0 & 0 & 0 & 0 \\
0 & 1 & 0 & 1 & \cdots & 0 & 0 & 0 & 0 & 0 \\
\vdots & \vdots & \vdots & \vdots & \cdots & \vdots & \vdots & \vdots & \vdots & \vdots \\
0 & 0 & 0 & 0 & \cdots & 1 & 0 & 1 & 0 & 0 \\
a_0 & a_1 & a_2 & a_3 & \cdots & a_{d-5} & a_{d-4}+1 & a_{d-3} & 0 & 0 \\
0 & 0 & 0 & 0 & \cdots & 0 & 0 & 1 & 1 & 0 \\
0 & 0 & 0 & 0 & \cdots & 0 & 0 & 0 & 0 & 1
\end{pmatrix}.
\]
Note that since the last row of $A+C$ is the last standard unit vector in $\IF_2^d$, we have that $A+C$ lies in $\GL_d(2)$ if and only if the $((d-1)\times(d-1))$-matrix $D$ obtained from $A+C$ by deleting the entries from the last column and row lies in $\GL_{d-1}(2)$. Now, we can choose the $b_i$ such that the first row of $D$ is the vector
\[
\begin{pmatrix}
0 & \cdots & 0 & 1 & 0 & 1
\end{pmatrix}
\in\IF_2^{d-1},
\]
and then as in the case \enquote{$M'=\{1,d-1\}$}, we see that $\im(D)=\IF_2^{d-1}$, i.e., that $D\in\GL_{d-1}(2)$, as required.

Finally, assume that $Q=X+1$. Then without loss of generality, we have
\[
A=
\begin{pmatrix}
1 & 1 & 0 & \cdots & 0 & 0 & 0 & 0 \\
0 & 1 & 1 & \cdots & 0 & 0 & 0 & 0 \\
\vdots & \vdots & \vdots & \cdots & \vdots & \vdots & \vdots & \vdots \\
0 & 0 & 0 & \cdots & 1 & 1 & 0 & 0 \\
0 & 0 & 0 & \cdots & 0 & 1 & 0 & 0 \\
0 & 0 & 0 & \cdots & 0 & 0 & 1 & 0 \\
0 & 0 & 0 & \cdots & 0 & 0 & 1 & 1
\end{pmatrix}.
\]
Make the ansatz
\[
C=\Comp(1+b_1X+b_2X^2+\cdots+b_{d-1}X^{d-1}+X^d)=
\begin{pmatrix}
0 & 1 & 0 & \cdots & 0 \\
0 & 0 & 1 & \cdots & 0 \\
\vdots & \vdots & \vdots & \cdots & \vdots \\
0 & 0 & 0 & \cdots & 1 \\
1 & b_1 & b_2 & \cdots & b_{d-1}
\end{pmatrix}
\]
with $b_1+\cdots+b_{d-1}=1$. Then
\[
A+C=
\begin{pmatrix}
1 & 0 & \cdots & 0 & 0 & 0 & 0 \\
0 & 1 & \cdots & 0 & 0 & 0 & 0 \\
\vdots & \vdots & \cdots & \vdots & \vdots & \vdots & \vdots \\
0 & 0 & \cdots & 1 & 0 & 0 & 0 \\
0 & 0 & \cdots & 0 & 1 & 1 & 0 \\
0 & 0 & \cdots & 0 & 0 & 1 & 1 \\
1 & b_1 & \cdots & b_{d-4} & b_{d-3} & b_{d-2}+1 & b_{d-1}+1
\end{pmatrix}.
\]
If we choose $b_i:=0$ for $i=1,2,\ldots,d-2$, and $b_{d-1}:=1$, then $A+C\in\GL_d(2)$, as required.
\end{enumerate}
\end{proof}

There is an interesting connection between the theory developed in this section and the recent group-theoretic paper \cite{LST20a} by Larsen, Shalev and Tiep. More precisely, in \cite[Section 10]{LST20a}, the following question is studied: \enquote{Which transitive permutation groups $S$ that are nonabelian simple as abstract groups have the property that every element of $S$ is a product of two derangements in $S$?}. Using tools from algebraic geometry and character theory, it is shown that $S$ has this property as long as $|S|$ is large enough, see \cite[Theorem 10.2]{LST20a}. Since $\GL_d(2)=\SL_d(2)=\PSL_d(2)$ is nonabelian simple for $d\geq3$, and since the elements of $\CGL_d(2)$ are the derangements of the natural transitive action of $\GL_d(2)$ on $\IF_2^d\setminus\{0\}$, the result \cite[Theorem 10.2]{LST20a} implies that $\CGL_d(2)^{(2)}=\GL_d(2)$ if $d$ is large enough. However, no explicit lower bound on $|S|$ is given in \cite[Theorem 10.2]{LST20a}, whence it is not clear whether \cite[Theorem 10.2]{LST20a} could be used to give a shorter proof of the fact that $\CGL_d(2)^{(2)}=\GL_d(2)$ for all $d\geq3$ than our elementary argument above.

In this context, we also note that our proof of Proposition \ref{cglProp} yields the following:

\begin{proposition}\label{derangementsProp}
Let $d$ be a positive integer, and let $q$ be a prime power. If $(d,q)\not=(1,2),(1,3),(2,2)$, then every element of $\GL_d(q)$ is a product of two fixed-point-free elements of $\GL_d(q)$ (i.e., derangements in the natural transitive action of $\GL_d(q)$ on $\IF_q^d\setminus\{0\}$).
\end{proposition}

\begin{proof}
If $q>2$, then our proof of Proposition \ref{cglProp} shows that $|\CGL_d(q)|>\frac{1}{2}|\GL_d(q)|$ -- note that this also applies in case $d=1$, although this was not noted explicitly in the proof of Proposition \ref{cglProp}. But $\CGL_d(q)$ is in bijection with the set $\OGL_d(q)$ of all fixed-point-free invertible $(d\times d)$-matrices of $\IF_q$ via the function $M\mapsto M+I_d$. Therefore, $|\OGL_d(q)|>\frac{1}{2}|\GL_d(q)|$, and an argument analogous to the one for \enquote{$d>1$ and $q>2$} in the proof of Proposition \ref{cglProp} shows that every element in $\GL_d(q)$ is a product of two elements of $\OGL_d(q)$, as required.

Now assume that $q=2$. Then $\OGL_d(q)=\CGL_d(q)$, and the result follows from Proposition \ref{cglProp}.
\end{proof}

Recall the cycle type sets $\Gamma(d,p,\ell)$ from Notation \ref{mainNot}(4). The following consequence of Proposition \ref{cglProp} exhibits its connection with Theorem \ref{mainTheo}:

\begin{corollary}\label{cglCor}
Let $p$ be a prime, and let $d$ and $\ell$ be positive integers. Then
\begin{equation}\label{corollaryEq}
\Gamma(d,p,\ell)=\{\CT(\lambda(M,w)): M\in\CGL_d(p)^{(\ell)}\text{ and }w\in\IF_p^d\}.
\end{equation}
\end{corollary}

\begin{proof}
This is clear for $\ell=1$ by the definitions of $\Gamma(d,p,\ell)$ and $\ACGL_d(p)$. If $\ell\geq2$ and $(d,p)\not=(1,2),(1,3),(2,2)$, then the set on the right-hand side of formula (\ref{corollaryEq}) is equal to $\CT(\AGL_d(p))=\Gamma(d,p,\ell)$ by Proposition \ref{cglProp}. Likewise, for $\ell\geq2$ and $(d,p)\in\{(1,2),(1,3),(2,2)\}$, one can check that the equality holds using the information on the set $\CGL_d(p)^{(\ell)}$ from Proposition \ref{cglProp} and Proposition \ref{fripertingerProp}.
\end{proof}

\section{Proof of Theorem \ref{mainTheo}}\label{sec4}

In this section, we develop the necessary theory for the proof of Theorem \ref{mainTheoExplicit} (a stronger version of Theorem \ref{mainTheo}), with which this section will be concluded. First, we discuss some generalities concerning coset-wise affine functions.

Let $K$ be a field, let $V$ be a finite-dimensional $K$-vector space, and let $W$ be a $K$-subspace of $V$. Throughout the rest of this discussion, we fix a complement $U$ of $W$ in $V$, so that $V=W\oplus U$. For $u\in U$, we set $W_u:=W+u$. Then the sets $W_u$ for the various $u\in U$ are the cosets of $W$ in $V$. For $U$-indexed families $\vec{v}=(v_u)_{u\in U}$ and $\vec{\varphi}=(\varphi_u)_{u\in U}$, of vectors in $V$ and $K$-endomorphisms of $V$ stabilizing $W$ respectively, we denote by $f_{\vec{\varphi},\vec{v}}$ the $W$-coset-wise $K$-affine function of $V$ such that for each $u\in U$, one has $f(x)=x^{\varphi_u}+v_u$ for all $x\in W_u$. Moreover, we denote by $g_{\vec{\varphi},\vec{v}}$ the unique function $U\rightarrow U$ such that for each $u\in U$, one has
\[
f(W_u)\subseteq W_{g_{\vec{\varphi},\vec{v}}(u)}.
\]

\begin{proposition}\label{basicProp}
With notation as fixed above, the following hold:
\begin{enumerate}
\item $f_{\vec{\varphi},\vec{v}}$ is a permutation of $V$ if and only if $\varphi_u$ restricts to an automorphism of $W$ for all $u\in U$ and $g_{\vec{\varphi},\vec{v}}$ is a permutation of $U$.
\item $f_{\vec{\varphi},\vec{v}}$ is a complete mapping of $V$ if and only if $\varphi_u$ restricts to a complete automorphism of $W$ for all $u\in U$ and $g_{\vec{\varphi},\vec{v}}$ is a complete mapping of $U$.
\end{enumerate}
\end{proposition}

\begin{proof}
For statement (1): If $f_{\vec{\varphi},\vec{v}}$ is a permutation of $V$, then the following must hold:
\begin{itemize}
\item $f_{\vec{\varphi},\vec{v}}$ must map each coset of $W$ onto a coset of $W$. For a fixed $u\in U$, the elements of $W_u$ are of the form $w+u$ where $w$ ranges over $W$, and their images under $f_{\vec{\varphi},\vec{v}}$ are of the form
\[
(w+u)^{\varphi_u}+v_u=w^{\varphi_u}+u^{\varphi_u}+v_u.
\]
If this is to assume all values in the coset $W_{u^{\varphi_u}+v_u}$, then $w^{\varphi_u}$ must assume all values in $W$. In other words, $\varphi_u$ must be a surjective $K$-endomorphism of $W$, i.e., a $K$-automorphism by the finiteness of $\dim_K(W)$.
\item $f_{\vec{\varphi},\vec{v}}$ must permute the cosets of $W$ in $V$. In other words, $g_{\vec{\varphi},\vec{v}}$ must be a permutation of $U$, as required.
\end{itemize}
Conversely, assume that each $\varphi_u$ restricts to an automorphism of $W$ and that $g_{\vec{\varphi},\vec{v}}$ is a permutation of $U$. The latter implies that every coset of $W$ in $V$ intersects $\im(f_{\vec{\varphi},\vec{v}})$, and the former that every coset of $W$ intersecting $\im(f_{\vec{\varphi},\vec{v}})$ is fully contained in $\im(f_{\vec{\varphi},\vec{v}})$. Together, this yields that $f_{\vec{\varphi},\vec{v}}$ is surjective and thus a permutation of $V$.

The proof of statement (2) is similar, and we omit it.
\end{proof}

The coset-wise $K$-affine functions $f_{\vec{\varphi},\vec{v}}$ of $V$ that are permutations of $V$ form a permutation group $\CWAff_W(V)$ on $V$, and we want to understand the structure of this permutation group. First, we make a simplification with regard to the involved endomorphisms $\varphi_u$.

Recall from above that for each $u\in U$ and each $w\in W$, we have
\[
f_{\vec{\varphi},\vec{v}}(w+u)=w^{\varphi_u}+u^{\varphi_u}+v_u.
\]
Assume that $\vec{\varphi'}=(\varphi'_u)_{u\in U}$ is a different family of $K$-endomorphisms of $V$ that stabilize $W$, and assume that for each $u\in U$, the restrictions of $\varphi_u$ and $\varphi'_u$ to $W$ are equal. If we set
\[
v'_u:=u^{\varphi_u}+v_u-u^{\varphi'_u}
\]
for each $u\in U$, and we set $\vec{v'}:=(v'_u)_{u\in U}$, then we find that
\[
f_{\vec{\varphi'},\vec{v'}}(w+u)=w^{\varphi'_u}+u^{\varphi'_u}+v'_u=w^{\varphi_u}+u^{\varphi_u}+v_u=f_{\vec{\varphi},\vec{v}}(w+u).
\]
This shows that we still get the full group $\CWAff_W(V)$ if we restrict to only such coset-wise $K$-affine permutations $f_{\vec{\varphi},\vec{v}}$ of $V$ where each $\varphi_u$ is an automorphism $\gamma_u$ of $V=W\oplus U$ of the form $\alpha_u\oplus\id_U$ for some $\alpha_u\in\Aut(W)$. We will henceforth assume that $\vec{\varphi}=\vec{\gamma}$ is chosen of this form.

For the formulation of the next theorem, we briefly recall the notion of an imprimitive permutational wreath product:

\begin{definition}\label{impWreathDef}
Let $G$ be an abstract group, and let $P\leq\Sym(\Lambda)$ be a permutation group.
\begin{enumerate}
\item The \emph{(abstract) wreath product of $G$ and $P$}, written $G\wr P$, is the abstract group that can be defined as the external semidirect product $P\ltimes G^{\Lambda}$ where $P$ acts on $G^{\Lambda}$, whose elements are $\Lambda$-indexed families $(g_{\lambda})_{\lambda\in\Lambda}$ of elements of $G$, by \enquote{coordinate permutations}. More explicitly, the elements of $G\wr P$ are ordered pairs of the form $(\sigma,\vec{g})$ with $\sigma\in P$ and $\vec{g}=(g_{\lambda})_{\lambda\in\Lambda}\in G^{\Lambda}$, and these elements are multiplied as follows:
\[
(\sigma,(g_{\lambda})_{\lambda\in\Lambda})\cdot(\psi,(h_{\lambda})_{\lambda\in\Lambda}):=(\sigma\psi,(g_{\psi^{-1}(\lambda)}h_{\lambda})_{\lambda\in\Lambda}).
\]
\item If $G\leq\Sym(\Omega)$ is a permutation group, then the abstract group $G\wr P$ is isomorphic to a certain permutation group on $\Omega\times\Lambda$ that is called the \emph{imprimitive (permutational) wreath product of $G$ and $P$} and will be denoted by $G\wr_{\imp}P$. More explicitly, the function $G\wr P\rightarrow\Sym(\Omega\times\Lambda)$ that maps $(\sigma,(g_{\lambda})_{\lambda\in\Lambda})\in G\wr P$ to the permutation
\[
(\omega,\lambda)\mapsto(g_{\sigma(\lambda)}(\omega),\sigma(\lambda))
\]
is an isomorphism of abstract groups between $G\wr P$ and $G\wr_{\imp}P$.
\end{enumerate}
\end{definition}

Moreover, we remind the reader that an \emph{isomorphism of permutation groups} $G\leq\Sym(\Omega)$ and $H\leq\Sym(\Lambda)$ is a bijection $\beta:\Omega\rightarrow\Lambda$ such that $\beta^{-1}G\beta=H$. In this case, the function $G\rightarrow H$, $g\mapsto\beta^{-1}g\beta$, is an isomorphism of abstract groups, and we say that $g\in G$ \emph{corresponds to} $\beta^{-1}g\beta\in H$ under $\beta$.

\begin{theorem}\label{wreathProdTheo}
Consider the bijection
\[
\iota:V=W\oplus U\rightarrow W\times U, w+u\mapsto (w,u).
\]
This bijection is a permutation group isomorphism between $\CWAff_W(V)$ and the imprimitive permutational wreath product $\Aff(W)\wr_{\imp}\Sym(U)$. In fact, if, as above, $\vec{\gamma}=(\gamma_u)_{u\in U}=(\alpha_u\oplus\id_U)_{u\in U}$ and $\vec{v}=(v_u)_{u\in U}$, and if we write $v_u=\omega_u+\nu_u$ with $\omega_u\in W$ and $\nu_u\in U$, then we have that under this isomorphism, the element $f_{\vec{\gamma},\vec{v}}$ of $\CWAff_W(V)$, permuting the cosets $W_u=W+u$ of $W$ in $V$ according to $g_{\vec{\gamma},\vec{v}}\in\Sym(U)$, corresponds to the wreath product element
\begin{equation}\label{wreathProductFormEq}
(g_{\vec{\gamma},\vec{v}},(\lambda(\alpha_u,\omega_u))_{u\in U}).
\end{equation}
\end{theorem}

\begin{proof}
For all $w\in W$ and all $u\in U$, we have
\begin{align*}
f_{\vec{\gamma},\vec{v}}(\iota^{-1}((w,u)))&=f_{\vec{\gamma},\vec{v}}(w+u)=w^{\alpha_u}+u+v_u=(w^{\alpha_u}+\omega_u)+(u+\nu_u) \\
&=(w^{\alpha_u}+\omega_u)+g_{\vec{\gamma},\vec{v}}(u)=\iota^{-1}((w^{\alpha_u}+\omega_u,g_{\vec{\gamma},\vec{v}}(u))) \\
&=\iota^{-1}((w,u)^{(g_{\vec{\gamma},\vec{v}},(\lambda(\alpha_u,\omega_u))_{u\in U})}),
\end{align*}
which shows that the diagram
\begin{center}
\begin{tikzpicture}
\matrix (m) [matrix of math nodes,row sep=3em,column sep=4em,minimum width=2em]
  {
     V & V \\
     W\times U & W\times U \\};
  \path[-stealth]
    (m-1-1) edge node[above,shift={(0pt,5pt)}] {$f_{\vec{\gamma},\vec{v}}$} (m-1-2)
    (m-1-1) edge node [left] {$\iota$} (m-2-1)
		(m-1-2) edge node [right] {$\iota$} (m-2-2)
		(m-2-1) edge node [below,shift={(0pt,-5pt)}] {$(g_{\vec{\gamma},\vec{v}},(\lambda(\alpha_u,\omega_u))_{u\in U})$} (m-2-2);
\end{tikzpicture}
\end{center}
is commutative. This shows that the group isomorphism $\Sym(V)=\Sym(W\oplus U)\rightarrow\Sym(W\times U)$, $\sigma\mapsto\iota^{-1}\sigma\iota$, restricts to an injective group homomorphism $\CWAff_W(V)\rightarrow\Aff(W)\wr_{\imp}\Sym(U)$, and it remains to show that this homomorphism is also surjective. Let $(\psi,(\lambda(\alpha_u,\omega_u))_{u\in U})\in\Aff(W)\wr_{\imp}\Sym(U)$. Set $\vec{\gamma}:=(\alpha_u\oplus\id_U)_{u\in U}$ and $\vec{v}:=(\omega_u+u^{\psi}-u)_{u\in U}$. The calculations from the beginning of this proof show that
\[
\iota\circ f_{\vec{\gamma},\vec{v}}\circ\iota^{-1}=\iota^{-1}f_{\vec{\gamma},\vec{v}}\iota=(\psi,(\lambda(\alpha_u,\omega_u))_{u\in U}),
\]
as required.
\end{proof}

The following result concerning cycle types in imprimitive permutational wreath products will be useful in view of Theorem \ref{wreathProdTheo}:

\begin{lemma}\label{polyaLem}
Let $\Omega$ and $\Lambda$ be finite sets, let $G\leq\Sym(\Omega)$ and $H\leq\Sym(\Lambda)$, and let $P:=G\wr_{\imp}H\leq\Sym(\Omega\times\Lambda)$. Consider an element $\sigma=(\psi,(g_{\lambda}))_{\lambda\in\Lambda}\in P$. For each cycle $\zeta=(\lambda_0,\lambda_1,\ldots,\lambda_{\ell-1})$ of $\sigma$ on $\Lambda$, call an element of $G$ of the form $g_{\lambda_0}g_{\lambda_1}\cdots g_{\lambda_{\ell-1}}$ a \emph{forward cycle product of $\sigma$ with respect to $\zeta$}. Since all forward cycle products of $\sigma$ with respect to $\zeta$ are $G$-conjugate to each other, the cycle type $\gamma_{\zeta}(\sigma):=\CT(g_{\lambda_0}\cdots g_{\lambda_{\ell-1}})$ is uniquely determined by $\sigma$ and $\zeta$. Moreover,
\[
\CT(\sigma)=\prod_{\zeta}{\BU_{\ell(\zeta)}(\gamma_{\zeta}(\sigma))},
\]
where $\zeta$ ranges over the cycles of $\psi$ on $\Lambda$, and $\ell(\zeta)$ denotes the length of $\zeta$.
\end{lemma}

\begin{proof}
This is a slightly more general version of \cite[Lemma 3.5]{BW21a}, a \enquote{local} version of P{\'o}lya's celebrated formula for the cycle index of an imprimitive permutational wreath product \cite[table at the bottom of p.~180]{Pol37a}. The proof is analogous to the one of \cite[Lemma 3.5]{BW21a}. Note that we used the notation $\CT^{(\ell)}(\alpha)$ for $\BU_{\ell}(\CT(\alpha))$ in \cite[Lemma 3.5]{BW21a}.
\end{proof}

We now specialize to $K=\IF_p$ for some prime $p$. Using the theory developed thus far, we can formulate and prove Theorem \ref{mainTheoExplicit} below. To make the formulation of the theorem itself more concise, we introduce some notation used in it.

Let $p$ be a prime, and let $d$ and $t$ be positive integers. Assume that $g$ is a complete mapping of $\IF_p^t=:U$ of cycle type $x_1^{k_1}\cdots x_{p^t}^{k_{p^t}}$. For $\ell=1,2,\ldots,p^t$, enumerate the length $\ell$ cycles of $g$ on $U$ as
\[
\zeta_{\ell,i}=(u_{\ell,i,0},u_{\ell,i,1},\ldots,u_{\ell,i,\ell})\text{ for }i=1,2,\ldots,k_{\ell}.
\]
Moreover, for $\ell=1,2,\ldots,p^t$ and $i=1,2,\ldots,k_{\ell}$, choose a cycle type $\gamma_{\ell,i}\in\Gamma(d,p,\ell)$. By Corollary \ref{cglCor}, we can write
\[
\gamma_{\ell,i}=\CT(\lambda(M_{\ell,i,0}\cdots M_{\ell,i,\ell-1},w_{\ell,i}))\in\CT(\AGL_d(p))
\]
for suitable $M_{\ell,i,0},\ldots,M_{\ell,i,\ell-1}\in\CGL_d(p)$ and $w_{\ell,i}\in\IF_p^d=:W$.

For $\ell=1,2,\ldots,p^t$, $i=1,2,\ldots,k_{\ell}$ and $j=0,1,\ldots,\ell-1$, set
\[
\alpha_{u_{\ell,i,j}}:=(w\mapsto wM_{\ell,i,j})\in\Aut_{\IF_p}(W),
\]
set
\[
\omega_{u_{\ell,i,j}}:=\begin{cases}\vec{0}\in\IF_p^d, & \text{if }j<\ell-1, \\ w_{\ell,i}, & \text{if }j=\ell-1,\end{cases}
\]
and set
\[
\nu_{u_{\ell,i,j}}:=u_{l,i,(j+1)\mod{\ell}}-u_{\ell,i,j}.
\]
Observe that this defines the notations $\alpha_u$, $\omega_u$ and $\nu_u$ uniquely for each $u\in U$. Consider the $\IF_p$-vector space $V=\IF_p^{d+t}=\IF_p^d\oplus\IF_p^t=W\oplus U$. Set
\[
\vec{\gamma}:=(\alpha_u\oplus\id_U)_{u\in U}\in\Aut_{\IF_p}(V)^U\text{ and }\vec{v}:=(\omega_u+\nu_u)_{u\in U}\in V^U.
\]

\begin{theorem}\label{mainTheoExplicit}
With notatation as fixed above, we have that the $W$-coset-wise $K$-affine function
\[
f_{\vec{\gamma},\vec{v}}:V\rightarrow V, v=w+u\mapsto w^{\alpha_u}+u+\omega_u+\nu_u,
\]
is a complete mapping of $V$ of cycle type
\[
\prod_{\ell=1}^{p^t}\prod_{i=1}^{k_{\ell}}{\BU_{\ell}(\gamma_{\ell,i})}.
\]
\end{theorem}

\begin{proof}
By Theorem \ref{wreathProdTheo}, $f_{\vec{\gamma},\vec{v}}$ corresponds under a suitable isomorphism of permutation groups $\CWAff_W(V)\rightarrow\Aff(W)\wr_{\imp}\Sym(U)$ to the wreath product element $(\psi,(\lambda(\alpha_u,\omega_u))_{u\in U})$ where $u^{\psi}=u+\nu_u$ for all $u\in U$. By definition of $\nu_u$, we have $u+\nu_u=u^g$, and so $f_{\vec{\gamma},\vec{v}}$ corresponds to $\sigma:=(g,(\lambda(\alpha_u,\omega_u))_{u\in U})$. By Lemma \ref{polyaLem}, it suffices to show that for all $\ell=1,2,\ldots,p^t$ and $i=1,2,\ldots,k_{\ell}$, the cycle type of any forward cycle product of $\sigma$ with respect to $\zeta_{\ell,i}$ is equal to $\gamma_{\ell,i}$. But by definition, the following is such a forward cycle product:
\begin{align*}
&\lambda(\alpha_{u_{\ell,i,0}},\omega_{u_{\ell,i,0}})\cdot\lambda(\alpha_{u_{\ell,i,1}},\omega_{u_{\ell,i,1}})\cdots\lambda(\alpha_{u_{\ell,i,\ell-1}},\omega_{u_{\ell,i,\ell-1}})= \\
&\lambda(\alpha_{u_{\ell,i,0}}\alpha_{u_{\ell,i,1}}\cdots\alpha_{u_{\ell,i,\ell-1}},\sum_{k=0}^{\ell-1}{\left(\prod_{j=k+1}^{\ell-1}{\alpha_{u_{\ell,i,j}}}\right)\omega_{u_{\ell,i,k}}})= \\
&\lambda(M_{\ell,i,0}M_{\ell,i,1}\cdots M_{\ell,i,\ell-1},w_{\ell,i}),
\end{align*}
and this has the cycle type $\gamma_{\ell,i}$ by construction.
\end{proof}

We remark that an analogous proof shows that if each $\gamma_{\ell,i}$ lies in $\CT(\AGL_d(p))$, but not necessarily in $\Gamma(d,p,\ell)$, then one obtains a permutation (but not necessarily a complete mapping) $f_{\vec{\gamma},\vec{v}}$ of $V$ of cycle type
\[
\prod_{\ell=1}^{p^t}\prod_{i=1}^{k_{\ell}}{\BU_{\ell}(\gamma_{\ell,i})}.
\]
by choosing matrices $M_{\ell,i,j}\in\GL_d(p)$ and vectors $w_{\ell,i}\in\IF_p^d$ such that
\[
\gamma_{\ell,i}=\lambda(M_{\ell,i,0}M_{\ell,i,1}\cdots M_{\ell,i,\ell-1},w_{\ell,i}).
\]

\section{Construction of one-cycle complete mappings}\label{sec5}

In this section, we discuss an exemplary application of Theorem \ref{mainTheoExplicit} -- constructing, for each odd prime power $q$, a permutation $f_q$ of $\IF_q$ of cycle type $x_q$ such that $f_q$ is a complete mapping of $\IF_q$. We remark that the construction also makes sense if $q$ is even and yields a permutation of $\IF_q$ of cycle type $x_q$ which is not a complete mapping of $\IF_q$.

Write $q=p^k$ with $k\geq1$. For fixed $p$, we recursively construct a complete mapping $h_k$ of the $\IF_p$-vector space $\IF_p^k$ of cycle type $x_{p^k}$ as follows:
\begin{itemize}
\item For $k=1$, let $h_1:\IF_p\rightarrow\IF_p$, $x\mapsto x+1$. Observe that $h_1$ is of cycle type $x_p$ and is a complete mapping of $\IF_p$ if $p>2$.
\item Now assume that $k>1$ and that we already defined a permutation $h_{k-1}$ of $\IF_p^{k-1}=:U$ such that $h_{k-1}$ is of cycle type $x_{p^{k-1}}$ and is a complete mapping of $U$ if $p>2$. We can write $h_{k-1}$ in cycle notation as
\[
(v_0,v_1,\ldots,v_{p^{k-1}-2},v_{p^{k-1}-1}=\vec{0}).
\]
Moreover, we write
\[
h_1=\lambda(1,1)=\lambda(1,0)\lambda(1,0)\cdots\lambda(1,0)\lambda(1,1)
\]
as a product with $\ell$ factors in $\AGL_1(p)$, each of which is a complete mapping of $\IF_p=:W$ if $p>2$. Following the proof of Theorem \ref{mainTheoExplicit}, if we set
\[
\vec{\gamma}:=(\id_W\oplus\id_U)_{u\in U}=(\id_{\IF_p^k})_{u\in U}
\]
and $\vec{v}:=(v_u)_{u\in U}$ with
\[
v_u:=
\begin{cases}
1_W+u^{h_{k-1}}-u, & \text{if }u=\vec{0}\in U, \\
u^{h_{k-1}}-u, & \text{if }u\not=\vec{0},
\end{cases}
\]
then the function $h_k:\IF_p^k=W\oplus U\rightarrow\IF_p^k$,
\begin{equation}\label{hRecEq}
w+u\mapsto w+u+v_u=\begin{cases}(w+1_W)+u^{h_{k-1}}, & \text{if }u=\vec{0}, \\ w+u^{h_{k-1}}, & \text{if }u\not=\vec{0},\end{cases}
\end{equation}
has cycle type $x_{p^k}$ and is a complete mapping of $\IF_p^k$ if $p>2$.
\end{itemize}

Using the recursive formula (\ref{hRecEq}), it is not hard to show by induction on $k$ that $h_k$ has the explicit form
\begin{equation}\label{hExplicitEq}
(x_1,\ldots,x_k)^{h_k}=(x_1,\ldots,x_{\ell-1},x_{\ell}+1,\ldots,x_k+1)\text{ if }x_k=x_{k-1}=\cdots=x_{\ell+1}=0.
\end{equation}
For example,
\[
(x,y)^{h_2}=\begin{cases}(x,y+1), & \text{if }y\not=0, \\ (x+1,y+1), & \text{if }y=0,\end{cases}
\]
and
\[
(x,y,z)^{h_3}=\begin{cases}(x,y,z+1), & \text{if }z\not=0, \\ (x,y+1,z+1), & \text{if }z=0\text{ and }y\not=0, \\ (x+1,y+1,z+1), & \text{if }y=z=0.\end{cases}
\]

We can also give a polynomial formula for a function $\IF_q\rightarrow\IF_q$ which, with regard to a suitable $\IF_p$-basis of $\IF_q$, has the form (\ref{hExplicitEq}) and thus is a permutation of $\IF_q$ of cycle type $x_q$ and a complete mapping of $\IF_q$ if $q$ is odd. This uses the following well-known elementary lemma:

\begin{lemma}\label{coordinateFunctionLem}
Let $q=p^k$ be a prime power, and let $\omega\in\IF_q$ be of algebraic degree $k$ over $\IF_p$, so that $\Bcal:=(\omega^i)_{i=0,1,\ldots,k-1}$ is an $\IF_p$-basis of $\IF_q$. For $i=0,1,\ldots,k-1$, denote by $\pi_i$ the function $\IF_q\rightarrow\IF_p\subseteq\IF_q$ which maps $x=\sum_{j=0}^{k-1}{x_j\omega^j}\in\IF_q$ to its $i$-th $\Bcal$-coordinate $x_i$. Then for all $x\in\IF_q$, we have
\begin{align}\label{coordinatesEq}
&\begin{pmatrix}
\notag \pi_0(x) & \pi_1(x) & \pi_2(x) & \cdots & \pi_{k-1}(x)
\end{pmatrix}
= \\
&\begin{pmatrix}
x & x^p & x^{p^2} & \cdots & x^{p^{k-1}}
\end{pmatrix}
\cdot
\begin{pmatrix}
1 & 1 & 1 & \cdots & 1 \\
\omega & \omega^p & \omega^{p^2} & \cdots & \omega^{p^{k-1}} \\
\omega^2 & \omega^{2p} & \omega^{2p^2} & \cdots & \omega^{2p^{k-1}} \\
\vdots & \vdots & \vdots  & \cdots & \vdots \\
\omega^{k-1} & \omega^{(k-1)p} & \omega^{(k-1)p^2} & \cdots & \omega^{(k-1)p^{k-1}}
\end{pmatrix}^{-1}.
\end{align}
\end{lemma}

Note that formula (\ref{coordinatesEq}) expresses each coordinate function $\pi_i$ as an $\IF_p$-linearized polynomial function. Now, the function $f_{p^k}:\IF_{p^k}\rightarrow\IF_{p^k}$ which with respect to the $\IF_p$-basis $\Bcal$ of $\IF_q$ has the form (\ref{hExplicitEq}) can be written as
\[
f_{p^k}(x)=
\begin{cases}
x+\omega^{k-1}, & \text{if }\pi_{k-1}(x)\not=0, \\
x+\omega^{k-1}+\omega^{k-2}, & \text{if }\pi_{k-1}(x)=0\text{ and }\pi_{k-2}(x)\not=0, \\
x+\omega^{k-1}+\omega^{k-2}+\omega^{k-3}, & \text{if }\pi_{k-1}(x)=\pi_{k-2}(x)=0\text{ and }\pi_{k-3}(x)\not=0, \\
\vdots & \vdots \\
x+\omega^{k-1}+\omega^{k-2}+\cdots+\omega+1, & \text{if }\pi_{k-1}(x)=\pi_{k-2}(x)=\cdots=\pi_1(x)=0.
\end{cases}
\]
We recursively define functions $g_j:\IF_{p^k}\rightarrow\IF_{p^k}$ for $j=1,2,\ldots,k-1$ as follows:
\begin{itemize}
\item $g_1(x):=1-\pi_1(x)^{p-1}$.
\item For $j=2,3,\ldots,k$: $g_j(x):=(1-\pi_j(x)^{p-1})(\omega^{j-1}+g_{j-1}(x))$.
\end{itemize}
It is not hard to show by induction on $j$ that
\[
g_j(x)=
\begin{cases}
0, & \text{if }\pi_j(x)\not=0, \\
\omega^{j-1}, & \text{if }\pi_j(x)=0\text{ and }\pi_{j-1}(x)\not=0, \\
\omega^{j-1}+\omega^{j-2}, & \text{if }\pi_j(x)=\pi_{j-1}(x)=0\text{ and }\pi_{j-2}(x)\not=0, \\
\vdots & \vdots \\
\omega^{j-1}+\omega^{j-2}+\cdots+\omega+1, & \text{if }\pi_j(x)=\pi_{j-1}(x)=\cdots=\pi_1(x)=0.
\end{cases}
\]
Hence
\[
f_{p^k}(x)=x+\omega^{k-1}+g_{k-1}(x)\text{ for all }x\in\IF_{p^k}.
\]
Since we know the reduced polynomial forms of the coordinate functions $\pi_i$ by formula (\ref{coordinatesEq}), we can recursively work out the reduced polynomial forms of the functions $g_j$, allowing us to compute the reduced polynomial form of $f_{p^k}$.

\begin{example}\label{f27Ex}
We compute a reduced polynomial over $\IF_{27}$ that represents a complete mapping of $\IF_{27}$ of cycle type $x_{27}$. The computations in this example were carried out using GAP \cite{GAP4}. Let $\omega\in\IF_{27}$ be a root of the Conway polynomial $X^3-X+1\in\IF_3[X]$. In particular, $\omega$ is a primitive root of $\IF_{27}$. By Lemma \ref{coordinateFunctionLem}, we have for all $x\in\IF_{27}$ that
\begin{align*}
&\begin{pmatrix}
\pi_0(x) & \pi_1(x) & \pi_2(x)
\end{pmatrix}= \\
&\begin{pmatrix}
x & x^3 & x^9
\end{pmatrix}
\cdot
\begin{pmatrix}
1 & 1 & 1 \\
\omega & \omega^3 & \omega^9 \\
\omega^2 & \omega^6 & \omega^{18}
\end{pmatrix}^{-1}= \\
&\begin{pmatrix}
x & x^3 & x^9
\end{pmatrix}
\cdot
\begin{pmatrix}
\omega^{25} & \omega^{14} & -1 \\
\omega^{23} & \omega^{16} & -1 \\
\omega^{17} & \omega^{22} & -1
\end{pmatrix}.
\end{align*}
Hence
\begin{align*}
&\pi_0(x)=\omega^{25}x+\omega^{23}x^3+\omega^{17}x^9, \\
&\pi_1(x)=\omega^{14}x+\omega^{16}x^3+\omega^{22}x^9, \\
&\pi_2(x)=-x-x^3-x^9.
\end{align*}
Observe that
\begin{align*}
&g_1(x)=1-\pi_1(x)^2= \\
&\omega^5x^{18}+\omega^{12}x^{12}+\omega^{10}x^{10}+\omega^{19}x^6+\omega^4x^4+\omega^{15}x^2+1
\end{align*}
and
\begin{align*}
&g_2(x)=(1-\pi_2(x)^2)\cdot(\omega+g_1(x))= \\
&\omega^{18}x^{36}+\omega^{10}x^{30}+\omega^{12}x^{28}+x^{24}+x^{22}+x^{20}+\omega^{16}x^{18}+x^{16}+x^{14}+\omega^9x^{12}+\omega^{23}x^{10}+ \\
&x^8+\omega^{16}x^6+\omega^{25}x^4+\omega^{19}x^2+\omega^9= \\
&x^{24}+x^{22}+x^{20}+\omega^{16}x^{18}+x^{16}+x^{14}+\omega^9x^{12}+\omega^9x^{10}+x^8+\omega^{16}x^6+\omega^9x^4+ \\
&\omega^{16}x^2+\omega^9.
\end{align*}
Therefore, we have
\begin{align*}
&f_{27}(x)=x+\omega^2+g_2(x)= \\
&x^{24}+x^{22}+x^{20}+\omega^{16}x^{18}+x^{16}+x^{14}+\omega^9x^{12}+\omega^9x^{10}+x^8+\omega^{16}x^6+\omega^9x^4+ \\
&\omega^{16}x^2+x+\omega^6.
\end{align*}
\end{example}

\section{Concluding remarks}\label{sec6}

In this paper, we were concerned with producing examples of cycle types of complete mappings of finite fields. That is, we exhibited elements of the set
\[
\{\CT(f): f\text{ is a complete mapping of }\IF_q\}.
\]
There is also a related, harder problem, which asks for simultaneous control over the cycle types of a complete mapping $f$ and its associated orthomorphism $f+\id$. In other words, this problem is concerned with the set
\[
\{(\CT(f),\CT(f+\id)): f\text{ is a complete mapping of }\IF_q\}
\]
of cycle type pairs. For the most basic examples of complete mappings, scalar multiplications $f_a:x\mapsto ax$ for a fixed $a\in\IF_q^{\ast}\setminus\{-1\}$, controlling the cycle types of $f_a$ and $f_a+\id=f_{a+1}$ simultaneously is tantamount to controlling the multiplicative orders of $a$ and $a+1$ simultaneously, for which a theorem of Carlitz, \cite[Theorem 1]{Car56a}, is an important tool.

However, with $f_a$ and $f_a+\id$ both being scalar multiplications, the range of possibilities for $(\CT(f_a),\CT(f_a+\id))$ is rather limited, and in order to obtain more interesting examples of such cycle type pairs, one needs to study a larger class of complete mappings $f$ such that simultaneous control over $\CT(f)$ and $\CT(f+\id)$ can be gained. In a follow-up paper, the authors intend to do so for a certain subclass of the class of complete, coset-wise $\IF_p$-affine mappings of $\IF_{p^k}$.

\end{document}